\documentclass{article}
\usepackage{graphicx} 
\graphicspath{{./images/}}
\usepackage{caption, subcaption}
\usepackage[a4paper,left=30mm,right=30mm,top=32mm,bottom=32mm]{geometry}
\usepackage{amssymb, mathtools, enumitem, amsthm, bbm, parskip}
\numberwithin{equation}{section}
\newtheorem{theorem}{Theorem}[section]
\newtheorem{lemma}[theorem]{Lemma}
\newtheorem{corollary}[theorem]{Corollary}
\newtheorem{assumption}[theorem]{Assumption}
\newtheorem*{remark}{Remark}

\title{Long-time behavior of logarithmic spiral vortex sheets with two branches}
\author{Minki Cho\thanks{Department of Mathematical Sciences, Seoul National University. E-mail: andbi@snu.ac.kr}}
\date{December 4, 2023}

\begin{document}

\maketitle

\begin{abstract}
    We consider logarithmic spiral vortex sheets consisting of two branches. Based on some simple assumptions that appear true by numerical computations, we fully classify their long-time behavior and asymptotics, where in all cases each branch decays to $0$ or blows up in finite time. Furthermore, we present illustrations determining which range of initial data corresponds to each case. We also determine the asymptotic stability of the symmetric and asymmetric self-similar spirals.
\end{abstract}

\section{Introduction}
\subsection{Logarithmic spiral vortex sheets}
Logarithmic spiral vortex sheets are the vortex sheets of 2D incompressible perfect fluids that are $M$ branches of logarithmic spirals. The case $M = 1$ with self-similarity was first considered by Prandtl \cite{Pra22} in 1922. Later, the symmetric self-similar spirals, where the $M$ spirals are distributed under rotational symmetry, were considered by Alexander \cite{Ale71} in 1971 and so named Alexander spirals. Despite the long history, it was only in 2021 that Cieślak et al. \cite{CKOa} provided sufficient conditions for the self-similar spirals to indeed give rise to weak solutions to the 2D incompressible Euler equation. The same authors also proved the existence of asymmetric self-similar spirals with $M$ branches where $M \in \{2,3,5,7,9\}$ in \cite{CKOb} and, in \cite{CKOc}, the linear instability of Alexander spirals with $M \geq 3$ as solutions to the Birkhoff-Rott equation \cite{Bir62, Rot56}, which has been a traditional method for analyzing vortex sheet evolution.

Jeong and Said \cite{JS} recently suggested a new framework for studying the logarithmic spiral vortex sheets. They considered vorticity solutions to the 2D incompressible Euler equation under logarithmic spiral scaling invariance, converting the equation into a nonlinear transport system on the one-dimensional torus. Then, they viewed the spiral vortex sheets as solutions to the new system consisting of a finite number of Dirac deltas, namely of the form $\sum_{j=0}^{M-1} I_j(t)\delta_{\theta_j}(t)$, and proved their well-posedness. Also, they provided a simple criterion for finite-time blowup of such solutions, namely that they blow up if and only if $\beta\sum I_j < 0$, where $\beta$ is a parameter determining the shape of the spirals $\theta - \beta\log r = \text{const.}$. Finally, they briefly discussed the case $M = 2$, where $I_1$, $I_2$, and $\theta = \theta_1 - \theta_2$ satisfy the ODE system
\begin{equation} \label{original}
    \left\{ \begin{aligned}
        I_1' &= 2K'(0)I_1^2 + 2K'(\theta)I_1I_2,\\
        I_2' &= 2K'(0)I_2^2 + 2K'(-\theta)I_1I_2,\\
        \theta' &= 2(K(0)-K(-\theta))I_1 + 2(K(\theta)-K(0))I_2.
    \end{aligned} \right.
\end{equation}
Here, $K(\theta)$ is defined as
\begin{equation} \label{K(x):def}
    K(\theta) = \frac{1}{4}\operatorname{Re}\left[\frac{e^{\frac{2(\beta-i)}{1+\beta^2}(\theta-\pi)}}{\sin\left(\frac{2\pi(1+i\beta)}{1+\beta^2}\right)}\right]
\end{equation}
for $0 \leq \theta \leq 2\pi$, and accordingly
\begin{equation} \label{K'(x):def}
    K'(\theta) = \frac{2\beta}{1+\beta^2}K(\theta) + \frac{1}{2(1+\beta^2)}\operatorname{Im}\left[\frac{e^{\frac{2(\beta-i)}{1+\beta^2}(\theta-\pi)}}{\sin\left(\frac{2\pi(1+i\beta)}{1+\beta^2}\right)}\right].
\end{equation}
They proved that a self-similar solution, where $\theta$ is a constant, uniquely exists if and only if
\begin{equation} \label{F(theta)}
    F(\theta) = K(0)(K'(-\theta)-K'(\theta)) + K(\theta)(K'(0)-K'(-\theta)) + K(-\theta)(K'(\theta)-K'(0))
\end{equation}
equals to zero, and that an asymmetric self-similar solution ($\theta \neq \pi$) exists for small enough $\beta$, consistent with the result of \cite{CKOb}.

\subsection{Main results}
In this paper, we focus on the ODE system (\ref{original}) to determine the long-time behavior of its solutions, which correspond to the logarithmic spiral vortex sheets with two branches. Our main conclusion is the following:

\begin{theorem}[Long-time behavior of solutions] \label{thm:longtime_original}
    Suppose $\beta > 0$ and $\beta \notin \{\beta_0,\beta^*,\beta_2,\beta_3\}$. Let $(I_1(t),I_2(t),\theta(t))$ be a solution to (\ref{original}) with initial data $(I_1(0),I_2(0),\theta(0))$ such that $I_1(0), I_2(0) \neq 0$ and $0 < \theta(0) < 2\pi$. Then, $\theta$ always converges to a constant. Meanwhile, for $t > 0$, the long-time behavior and asymptotics of $(I_1,I_2)$ correspond to one of the following (up to symmetry):
    \begin{itemize}
        \item $I_1, I_2 \to 0$ (not both negative) as $t \to +\infty$ and they decay as $O(1/t)$.
        \item $I_1 \to 0$ (may be negative), $I_2 \to +0$ as $t \to +\infty$. $I_2$ decays as $O(1/t)$, but $I_1$ decays as about $O(1/t^{c_1})$ for some constant $c_1=c_1(\beta)>1$.
        \item $I_1 \to -\infty$, $I_2 \to 0$ as $t \to t^*$ for some $t^* < +\infty$. $I_1$ blows up as $O(1/(t^*-t))$, while $I_2$ decays as about $O((t^*-t)^c)$ for some constant $0<c<1$ depending on $\beta$ and $\theta(0)$. (There are only finite options of $c$ when $\beta$ is fixed.)
        \item $I_1,I_2 \to -\infty$ as $t \to t^*$ for some $t^* < +\infty$ and they blow up as $O(1/(t^*-t))$.
        \item $I_1 \to -\infty$, $I_2 \to \pm\infty$ as $t \to t^*$ for some $t^* < +\infty$. $I_1$ blows up as $O(1/(t^*-t))$, but $I_2$ blows up as about $O(1/(t^*-t)^{c_2})$ for some constant $c_2=c_2(\beta) \in (0,1)$.
    \end{itemize}
    (By ``$I_j$ decays/blows up as about $O(r(t))$'' we mean that $\log I_j \sim \log r(t)$.) For $t < 0$, the long-time behavior is obtained by that of $(-I_1,-I_2,\theta)$ with time reversal ($t > 0$). For $\beta < 0$, the system is equivalent to that with $\beta \mapsto -\beta > 0$ and time reversal.
\end{theorem}
\begin{remark}
    The constants $\beta_0$, $\beta^*$, $\beta_2$, $\beta_3$, $c_1$, and $c_2$ are specified later on.
\end{remark}

Furthermore, we illustrate how the entire phase space is partitioned into regions corresponding to each case in Theorem \ref{thm:longtime_original}. These results upgrade \cite[Theorem~1.11]{JS} for the case $N = 2$ in that we determine the long-time behavior of each of $I_1$ and $I_2$ for all cases and specify the decay and blowup rates. Also, we reveal the existence of the cases $(I_1 \downarrow 0, I_2 \uparrow 0)$ and $(I_1 \uparrow +\infty, I_2 \downarrow -\infty)$ (and the symmetric ones), which is unanticipated since the general behavior is that $(I_1 \downarrow 0, I_2 \downarrow -\infty)$ when $\beta > 0$ and $I_1 > 0, I_2 < 0$, influenced by the fact that $K'(0)$ is a large negative number when $\beta > 0$ (See Lemma \ref{lemma:K'(0)<0}).

We first discuss the properties of $K$ we need in Section \ref{section:prop}, where some are left as assumptions while appearing true. In Section \ref{section:reparam}, we reparametrize the time variable of (\ref{original}) to obtain a two-dimensional ODE system, which is crucial since we can then apply the Poincaré-Bendixson theorem. We also investigate how the symmetry of (\ref{original}) is involved in the new ODE system. In Section \ref{section:equil}, we analyze the nullclines and equilibrium points of the new system and the local behavior at those points. Especially, we determine the asymptotic stability of the symmetric and asymmetric self-similar spirals. Subsequently, in Section \ref{section:longtime}, we prove the non-existence of cycles based on the nullcline analysis and employ the Poincaré-Bendixson theorem. Furthermore, we construct a topological graph with heteroclinic orbits to illustrate the partition of the whole phase plane. Finally, we recover the original solutions for (\ref{original}) from the reparametrized ones and determine their asymptotic behavior.

\textbf{Acknowledgments.} The author gives special thanks to In-Jee Jeong for suggesting the problem and providing various helpful discussions. This work was supported by the Undergraduate Research Internship (Fall 2023) through the College of Natural Sciences, Seoul National University.

\section{Properties of the kernel} \label{section:prop}
In this section, we discuss some properties of $K$. First of all, we note that
\begin{equation} \label{-beta}
    K(\theta) = \frac{1}{4}\operatorname{Re}\left[\frac{e^{\frac{2(\beta-i)}{1+\beta^2}(\theta-\pi)}}{\sin\left(\frac{2\pi(1+i\beta)}{1+\beta^2}\right)}\right]
    = \frac{1}{4}\operatorname{Re}\left[\frac{e^{\frac{2(\beta+i)}{1+\beta^2}(\theta-\pi)}}{\sin\left(\frac{2\pi(1-i\beta)}{1+\beta^2}\right)}\right]
    = \Tilde{K}(-\theta),
\end{equation}
where $\Tilde{K}$ denotes the kernel $K$ with $-\beta$ instead of $\beta$. (The second equality is obtained by taking complex conjugates.) Therefore, the ODE system (\ref{original}) with $-\beta$ is simply a time reversal of that with $\beta$. As we will later deal with the time-reversed system, we shall consider only the case $\beta > 0$ from now on.

From (\ref{K(x):def}) and (\ref{K'(x):def}), we can see that $K$ itself is periodic but $K'$ is not. So we define \[K'(+0) = \lim_{\theta\to0+}K'(\theta), \quad K'(-0) = \lim_{\theta\to0-}K'(\theta).\] Then $K'(0)$ in (\ref{original}) is defined as \[K'(0) = \frac{K'(+0)+K'(-0)}{2}.\]

\begin{lemma} \label{lemma:K'(0)<0}
    For all $\beta > 0$ and $0 < \alpha < 2\pi$, the following inequalities hold.
    \begin{enumerate}[label=(\roman*)]
        \item $K'(0) < 0$.
        \item $|K'(\alpha) + K'(-\alpha)| < -2K'(0)$.
        \item $|K'(\pi)| < -K'(0)$.
    \end{enumerate}
\end{lemma}
\begin{proof}
By \cite[Lemma~2.2]{JS}, \[K'(0) = -4\beta\int K'(\theta)^2 d\theta\] and \[K'(\alpha) + K'(-\alpha) = -4\beta\int K'(\theta)(K'(\theta+\alpha)+K'(\theta-\alpha)) d\theta = -8\beta\int K'(\theta)K'(\theta+\alpha) d\theta.\] Now (i) is clear. Also, Cauchy-Schwarz inequality yields \[\left|\int K'(\theta)K'(\theta+\alpha) d\theta\right| < \left(\int K'(\theta)^2 d\theta\right)^\frac{1}{2} \left(\int K'(\theta+\alpha)^2 d\theta\right)^\frac{1}{2} = \int K'(\theta)^2 d\theta,\] which implies (ii). Taking $\alpha = \pi$ gives (iii).
\end{proof}

To simplify (\ref{K(x):def}) and (\ref{K'(x):def}), define \[\frac{2\pi(\beta-i)}{1+\beta^2} = z\] and \[\frac{\theta}{\pi} = 2k, \quad k \in [0,1].\] Then
\begin{equation} \label{K(x):simple}
    K(\theta) = \frac{1}{4}\operatorname{Re}\left[\frac{e^{2kz}e^{-z}}{\sin(iz)}\right] = \frac{1}{2}\operatorname{Im}\left[\frac{e^{2kz}}{e^{2z}-1}\right]
\end{equation}
and
\begin{equation} \label{K'(x):simple}
    K'(\theta) = \frac{\beta}{1+\beta^2}\operatorname{Im}\left[\frac{e^{2kz}}{e^{2z}-1}\right] - \frac{1}{1+\beta^2}\operatorname{Re}\left[\frac{e^{2kz}}{e^{2z}-1}\right].
\end{equation}

\begin{lemma} \label{lemma:K'(+0)>K'(-0)}
    For all $\beta > 0$, \[K'(+0)-K'(-0) = \frac{1}{1+\beta^2}.\] Consequently, $K'(-0) < K'(+0) < -K'(-0)$.
\end{lemma}
\begin{proof}
    From (\ref{K'(x):simple}), we have \[K'(+0) = \frac{\beta}{1+\beta^2}\operatorname{Im}\left[\frac{1}{e^{2z}-1}\right]-\frac{1}{1+\beta^2}\operatorname{Re}\left[\frac{1}{e^{2z}-1}\right]\] and
    \begin{align*}
        K'(-0) &= \frac{\beta}{1+\beta^2}\operatorname{Im}\left[\frac{e^{2z}}{e^{2z}-1}\right]-\frac{1}{1+\beta^2}\operatorname{Re}\left[\frac{e^{2z}}{e^{2z}-1}\right]\\
        &= \frac{\beta}{1+\beta^2}\operatorname{Im}\left[\frac{1}{e^{2z}-1}\right]-\frac{1}{1+\beta^2}\left(\operatorname{Re}\left[\frac{1}{e^{2z}-1}\right]+1\right)\\
        &= K'(+0)-\frac{1}{1+\beta^2}.
    \end{align*}
    Thus $K'(-0) < K'(+0)$. Combining this with Lemma \ref{lemma:K'(0)<0} yields $K'(-0) < 0$ and $K'(+0) < |K'(-0)|$.
\end{proof}

\begin{lemma} \label{lemma:K(x)=K(0)}
    There exist constants $0 < \beta_0 < \beta_2 < \beta_3$ such that the number of solutions to $K(\theta) = K(0)$ in $(0,2\pi)$ is
    \begin{itemize}
        \item $3$, if $\beta \in (0,\beta_0)$;
        \item $2$, if $\beta \in [\beta_0,\beta_2)$;
        \item $1$, if $\beta \in [\beta_2,\beta_3)$;
        \item $0$, if $\beta \in [\beta_3,\infty)$.
    \end{itemize}
\end{lemma}
\begin{proof}
    By (\ref{K(x):simple}), $K(\theta) = K(0)$ is equivalent to
    \begin{equation} \label{real}
        \frac{e^{2kz}-1}{e^{2z}-1} \in \mathbb{R}.
    \end{equation}
    From the definition of $z$,
    \begin{align*}
        e^{2z} &= e^{\frac{4\pi\beta}{1+\beta^2}}\left(\cos\left(\frac{4\pi}{1+\beta^2}\right)-i\sin\left(\frac{4\pi}{1+\beta^2}\right)\right)\\
        &= e^{\gamma\beta}(\cos\gamma-i\sin\gamma),
    \end{align*}
    where $\gamma$ is defined as \[\gamma = \frac{4\pi}{1+\beta^2} \in (0,4\pi).\] Then, assuming $\gamma \neq n\pi$($n\in\mathbb{Z}$), (\ref{real}) is equivalent to \[\frac{e^{k\gamma\beta}\cos(k\gamma)-1}{e^{k\gamma\beta}\sin(k\gamma)} = \frac{e^{\gamma\beta}\cos\gamma-1}{e^{\gamma\beta}\sin\gamma}.\] Denote the left-hand side as $f(k)$. That is, \[f(k) = \cot(k\gamma) - e^{-k\gamma\beta}\csc(k\gamma)\] for $k \in [0,1]$ and $k\gamma \neq n\pi$($n \in \mathbb{Z}$). Now we have to determine the number of solutions to $f(k) = f(1)$ in $(0,1)$. First, we note that $f'(k)<0$. Indeed, \[f'(k) = \gamma\csc^2(k\gamma)e^{-k\gamma\beta}\left(-e^{k\gamma\beta}+\beta\sin(k\gamma)+\cos(k\gamma)\right)\] and it is straightforward to confirm \[\beta\sin(k\gamma)+\cos(k\gamma) < e^{k\gamma\beta}.\] Also, \[\lim_{k\to0+}f(k) = \lim_{k\to0+}\frac{e^{k\gamma\beta}(\cos(k\gamma)-1)+e^{k\gamma\beta}-1}{e^{k\gamma\beta}\sin(k\gamma)} = \beta.\] Thus, while $k$ moves from $0$ to $1$, the value of $f$ starts from $\beta$ and decreases, with `jumps' from $-\infty$ to $+\infty$ whenever $k\gamma = n\pi$, and ends at $f(1)$. The number of solutions to $f(k) = f(1)$ depends on the number of jumps and the sign of $f(1) - \beta$, which we denote as $g(\beta)$. That is, \[g(\beta) = \cot\gamma - e^{-\gamma\beta}\csc\gamma - \beta.\] We can observe that \[g'(\beta) = \gamma\csc^2\gamma \cdot \frac{2\beta}{1+\beta^2} \cdot e^{-\gamma\beta} \left(e^{\gamma\beta} + \frac{1-\beta^2}{2\beta} - \cos\gamma\right) - 1\] is positive for all $\beta \in (0,\frac{1}{\sqrt{3}}) \cup (\frac{1}{\sqrt{3}},1) \cup (1,\sqrt{3})$. Indeed, if $0 < \beta < \frac{1}{\sqrt{7}}$($\frac{7}{2}\pi < \gamma < 4\pi$), we have \[\frac{1+\beta^2}{2\gamma\beta}\sin^2\gamma = \frac{(1+\beta^2)^2}{8\pi\beta}\sin^2\left(4\pi\cdot\frac{\beta^2}{1+\beta^2}\right) < \frac{(1+\beta^2)^2}{8\pi\beta} \left(4\pi\cdot\frac{\beta^2}{1+\beta^2}\right)^2 = 2\pi\beta^3 < \frac{1}{\sqrt{7}}\] and $\frac{1-\beta^2}{2\beta} - \cos\gamma > 0$. If $\frac{1}{\sqrt{7}} \leq \beta < \sqrt{3}$($\pi < \gamma \leq \frac{7}{2}\pi$), we have \[\frac{1+\beta^2}{2\beta} \cdot \frac{\sin^2\gamma}{\gamma} < \frac{4}{\sqrt{7}} \cdot \frac{1}{4} = \frac{1}{\sqrt{7}},\] $\gamma\beta \geq \frac{\sqrt{7}}{2}\pi$, and $\frac{1-\beta^2}{2\beta} - \cos\gamma > -\frac{1}{\sqrt{3}} - 1$. In both cases we get \[e^{\gamma\beta} + \frac{1-\beta^2}{2\beta} - \cos\gamma > \frac{1+\beta^2}{2\gamma\beta}\sin^2\gamma \cdot e^{\gamma\beta},\] which is equivalent to $g'(\beta) > 0$.
    Now, we separate cases by the number of jumps.
    \begin{enumerate}[label=(\roman*)]
        \item $0 < \gamma < \pi$($\beta > \sqrt{3}$). There is no jump, and $f$ just decreases from $\beta$ to $f(1)$. So, in this case, there is no solution to $f(k) = f(1)$ in $(0,1)$.
        \item $\pi < \gamma < 2\pi$($1 < \beta < \sqrt{3}$). There is a jump when $k\gamma = \pi$. Since $g(\beta)$ is increasing from $-\infty$ to $+\infty$ in this interval, there is a unique $\beta_3$ such that $g(\beta_3) = 0$. If $1 < \beta < \beta_3$, then $f(1) < \beta$ so there is exactly one solution to $f(k) = f(1)$ in $(0,1)$, which lies in $(0,\pi/\gamma)$. If $\beta_3 \leq \beta < \sqrt{3}$, then $f(1) \geq \beta$ so there is no solution.
        \item $2\pi < \gamma < 3\pi$($\frac{1}{\sqrt{3}} < \beta < 1$). There are two jumps when $k\gamma = \pi, 2\pi$. Since $g(\beta)$ is increasing from $-\infty$ to $+\infty$ in this interval, there is a unique $\beta_2$ such that $g(\beta_2) = 0$. If $\frac{1}{\sqrt{3}} < \beta < \beta_2$, then $f(1) < \beta$ so there are two solutions to $f(k) = f(1)$ in $(0,1)$, which lie in $(0,\pi/\gamma)$ and $(\pi/\gamma,2\pi/\gamma)$, respectively. If $\beta_2 \leq \beta < 1$, then $f(1) \geq \beta$ so there is one solution, which lies in $(\pi/\gamma,2\pi/\gamma)$.
        \item $3\pi < \gamma < 4\pi$($0 < \beta < \frac{1}{\sqrt{3}}$). There are three jumps when $k\gamma = \pi, 2\pi, 3\pi$. Since $g(\beta)$ is increasing from $-\infty$ to $+\infty$ in this interval, there is a unique $\beta_0$ such that $g(\beta_0) = 0$. If $0 < \beta < \beta_0$, then $f(1) < \beta$ so there are three solutions to $f(k) = f(1)$ in $(0,1)$, which lie in $(0,\pi/\gamma)$, $(\pi/\gamma,2\pi/\gamma)$, and $(2\pi/\gamma,3\pi/\gamma)$, respectively. If $\beta_0 \leq \beta < \frac{1}{\sqrt{3}}$, then $f(1) \geq \beta$ so there are two solutions, which lie in $(\pi/\gamma,2\pi/\gamma)$ and $(2\pi/\gamma,3\pi/\gamma)$, respectively.
    \end{enumerate}
    Finally, the cases $\gamma = n\pi$($n=1,2,3$) can be simply handled.
\end{proof}

\begin{remark}
    According to numerical simulation, the approximate values of $\beta_0$, $\beta_2$, and $\beta_3$ are $0.44$, $0.87$, and $1.55$, respectively.
\end{remark}

Thus, we denote the solutions to $K(\theta) = K(0)$ as
\begin{itemize}
    \item $0 < \theta_1 < \theta_2 < \theta_3 < 2\pi$ if $\beta \in (0,\beta_0)$;
    \item $0 < \theta_2 < \theta_3 < 2\pi$ if $\beta \in [\beta_0,\beta_2)$;
    \item $0 < \theta_3 < 2\pi$ if $\beta \in [\beta_2,\beta_3)$.
\end{itemize}

Suppose $\beta > 0$ and $\beta \notin \{\beta_0, \beta_2, \beta_3\}$. By Rolle's theorem, at least one $\theta$ with $K'(\theta) = 0$ is between each pair of consecutive solutions. But in view of (\ref{K'(x):simple}), $K'(\theta) = 0$ is equivalent to \[\tan\arg\left(\frac{e^{2kz}}{e^{2z}-1}\right) = \tan(-k\gamma - \arg(e^{2z}-1)) = \frac{1}{\beta},\] following the previous definitions of $k$ and $\gamma$. Thus, the solutions to $K'(2k\pi) = 0$ must appear at intervals of $\pi/\gamma$. The solutions to $K''(2k\pi) = 0$ also have the same property since
\begin{align*}
    K''(\theta) = 0 \Longleftrightarrow{}& K(\theta) = \beta K'(\theta)\\
    \Longleftrightarrow{}& (1-\beta^2)\operatorname{Im}\left[\frac{e^{2kz}}{e^{2z}-1}\right] = -2\beta\operatorname{Re}\left[\frac{e^{2kz}}{e^{2z}-1}\right].
\end{align*}
These facts imply that there is exactly one $\theta$ with $K'(\theta) = 0$ between each pair of consecutive solutions to $K(\theta) = K(0)$, and the sign of $K'$ must change whenever it passes through such $\theta$. So, we can derive the following lemma.

\begin{lemma} \label{lemma:K'(x):sign}
    For $\beta > 0$, the sign of $K'(\theta)$ when $K(\theta) = K(0)$ is determined as follows:
    \begin{itemize}
        \item $K'(+0) < 0$ if $\beta \in (0,\beta_0)\cup(\beta_2,\beta_3)$; $K'(+0) > 0$ if $\beta \in (\beta_0,\beta_2)\cup(\beta_3,\infty)$; $K'(+0) = 0$ if $\beta \in \{\beta_0,\beta_2,\beta_3\}$.
        \item $K'(\theta_1) > 0$.
        \item $K'(\theta_2) < 0$.
        \item $K'(\theta_3) > 0$.
        \item $K'(-0) < 0$. (Lemma \ref{lemma:K'(+0)>K'(-0)})
    \end{itemize}
    Consequently, the sign of $K(\theta)-K(0)$ is determined as follows:
    \begin{itemize}
        \item If $\beta \in (0,\beta_0)$, $K(\theta)-K(0) > 0$ when $\theta \in (\theta_1,\theta_2)\cup(\theta_3,2\pi)$; $K(\theta)-K(0) < 0$ when $\theta \in (0,\theta_1)\cup(\theta_2,\theta_3)$.
        \item If $\beta \in [\beta_0,\beta_2)$, $K(\theta)-K(0) > 0$ when $\theta \in (0,\theta_2)\cup(\theta_3,2\pi)$; $K(\theta)-K(0) < 0$ when $\theta \in (\theta_2,\theta_3)$.
        \item If $\beta \in [\beta_2,\beta_3)$, $K(\theta)-K(0) > 0$ when $\theta \in (\theta_3,2\pi)$; $K(\theta)-K(0) < 0$ when $\theta \in (0,\theta_3)$.
        \item If $\beta \in [\beta_3,\infty)$, $K(\theta)-K(0) > 0$ for all $\theta \in (0,2\pi)$.
    \end{itemize}
\end{lemma}

Lemma \ref{lemma:K'(x):sign} can be well understood with Figure \ref{fig:K(x)}, the graphs of $K$ for $\beta = 0.3$, $0.6$, $1.2$, and $2$. These values are contained in $(0,\beta_0)$, $(\beta_0,\beta_2)$, $(\beta_2,\beta_3)$, and $(\beta_3,\infty)$, respectively.

\begin{figure}[t]
    \centering
    \includegraphics[width=12cm]{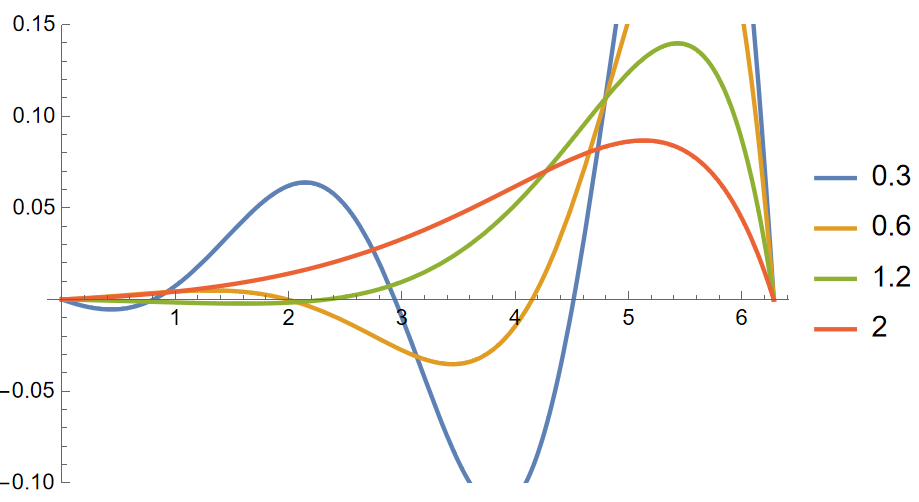}
    \caption{Graphs of $K(\theta)-K(0)$ for $\beta = 0.3$, $0.6$, $1.2$, and $2$}
    \label{fig:K(x)}
\end{figure}

Now, we further assume some properties of $K$. Although not proven formally, numerical simulation strongly suggests they are true.

First, note that $\theta_1 \to \frac{\pi}{2}$, $\theta_2 \to \pi$, $\theta_3 \to \frac{3\pi}{2}$ as $\beta \to 0$, and $\theta_3 = \pi$ when $\beta = 1$.

\begin{assumption} \label{assumption:xi}
    The angles $\theta_1$, $\theta_2$, $\theta_3$ are decreasing functions of $\beta > 0$. Also, $\theta_2 + \theta_3 > 2\pi$ when $\beta = \beta_0$. Thus, there exists a unique $\beta_1 \in (\beta_0,\beta_2)$ such that $\theta_2 + \theta_3 = 2\pi$. Meanwhile, $\theta_3 > \pi$ when $\beta = \beta_2$. Consequently, the following hold:
    \begin{itemize}
        \item If $\beta \in (0,\beta_0)$, then $\theta_1 + \theta_3 < 2\pi < \theta_2 + \theta_3$ and $\theta_2 < \pi$. Thus, \[0 < \theta_1 < 2\pi-\theta_3 < \theta_2 < 2\pi-\theta_2 < \theta_3 < 2\pi-\theta_1 < 2\pi.\]
        \item If $\beta \in (\beta_0,\beta_1)$, then $\theta_2 + \theta_3 > 2\pi$ and $\theta_2 < \pi$. Thus, \[0 < 2\pi-\theta_3 < \theta_2 < 2\pi-\theta_2 < \theta_3 < 2\pi.\]
        \item If $\beta \in (\beta_1,\beta_2)$, then $\theta_2 + \theta_3 < 2\pi$ and $\theta_3 > \pi$. Thus, \[0 < \theta_2 < 2\pi-\theta_3 < \theta_3 < 2\pi-\theta_2 < 2\pi.\]
        \item If $\beta \in (\beta_2,1)$, then $\theta_3 > \pi$. Thus, \[0 < 2\pi-\theta_3 < \theta_3 < 2\pi.\]
        \item If $\beta \in (1,\beta_3)$, then $\theta_3 < \pi$. Thus, \[0 < \theta_3 < 2\pi-\theta_3 < 2\pi.\]
    \end{itemize}
\end{assumption}

\begin{remark}
    According to numerical simulation, the approximate value of $\beta_1$ is $0.57$.
\end{remark}

Next, in view of (\ref{K'(x):simple}), we can see that generally the size of $K'(\theta) = K'(2k\pi)$ gets larger when $k$ increases by $\pi/\gamma$.

\begin{assumption} \label{assumption:K'(x)>K'(0)}
    For each $\beta > 0$, there exists a unique $\alpha\in(\pi,2\pi)$ satisfying the following:
    \begin{itemize}
        \item For $\theta\in(0,\alpha)$, $K'(\theta) > K'(0)$.
        \item $K'(\alpha) = K'(0)$.
        \item For $\theta\in(\alpha,2\pi)$, $K'(\theta) < K'(0)$.
    \end{itemize}
    Consequently, $K''(\alpha) < 0$.
\end{assumption}

Since there is no $\theta \in (\alpha,2\pi)$ such that $K'(\theta) = 0$, it follows that $\theta_3 < \alpha < 2\pi$ whenever $0 < \beta < \beta_3$.

Finally, note that the function $F$ specified in (\ref{F(theta)}) satisfies $F(0) = F(\pi) = F(2\pi) = 0$ and $F(\theta) = -F(2\pi-\theta)$ for all $0 \leq \theta \leq 2\pi$. In other words, the graph of $F$ is symmetric with respect to the point $(\pi,0)$.

\begin{assumption} \label{assumption:F(x)}
    There is a unique $\beta^*>0$ such that
    \begin{equation} \label{F'(pi)}
        F'(\pi) = 2K'(\pi)(K'(0)-K'(\pi)) - 2K''(\pi)(K(0)-K(\pi))
    \end{equation}
    equals to zero. Moreover, $F'(\pi)>0$ if $0 < \beta < \beta^*$; $F'(\pi)<0$ if $\beta > \beta^*$. Consequently, when $0 < \beta < \beta^*$, there exists a unique $\Bar{\theta} \in (0,\pi)$ such that $F(\Bar{\theta}) = 0$ (and then $F'(\Bar{\theta}) < 0$); if $\beta > \beta^*$, there is no such $\Bar{\theta}$. Meanwhile, $\beta^* \in (\beta_1,\beta_2)$.
\end{assumption}

\begin{remark}
    According to numerical simulation, the approximate value of $\beta^*$ is $0.71$.
\end{remark}

\begin{assumption} \label{assumption:slope}
    For all $0 < \beta < \beta_1$ and $\theta,\eta \in (\theta_2,2\pi-\theta_2)$, the following holds: \[\left|\frac{K''(\theta)}{K'(\theta)-K'(0)} + \frac{K''(-\theta)}{K'(-\theta)-K'(0)}\right| < \frac{K'(\eta)}{K(\eta)-K(0)} + \frac{K'(-\eta)}{K(-\eta)-K(0)}.\]
    Moreover, for all $\beta_1 < \beta < \beta^*$ and $\theta \in (\Bar{\theta},2\pi-\Bar{\theta})$, the following holds: \[\frac{K''(\theta)}{K'(\theta)-K'(0)} + \frac{K''(-\theta)}{K'(-\theta)-K'(0)} > 0.\]
\end{assumption}

\section{Time reparametrization and symmetry} \label{section:reparam}
Recall the original ODE system (\ref{original}). We clearly see that $\{I_1 = 0\}$ and $\{I_2 = 0\}$ are invariant sets, so the signs of $I_1$ and $I_2$ remain unchanged. Let us assume $I_2 \neq 0$ and write $I_1/I_2 = R$. Then \[R' = \frac{I_1'I_2-I_1I_2'}{I_2^2} = I_2\bigl[2(K'(0)-K'(-\theta))R^2+2(K'(\theta)-K'(0))R\bigr].\]
If we assume $I_2>0$, then we can reparametrize the time variable by \[s = \phi(t) = \int_0^t I_2(u)du\] so that $ds/dt = I_2(t)$.
Now, on the range of $\phi$, define $\Tilde{R} = R\circ\phi^{-1}$ and $\Tilde{\theta} = \theta\circ\phi^{-1}$. Then, from (\ref{original}) we have
\begin{equation} \label{reparam}
    \left\{ \begin{aligned}
        \Tilde{R}' &= 2(K'(0)-K'(-\Tilde{\theta}))\Tilde{R}^2 + 2(K'(\Tilde{\theta})-K'(0))\Tilde{R},\\
        \Tilde{\theta}' &= 2(K(0)-K(-\Tilde{\theta}))\Tilde{R} + 2(K(\Tilde{\theta})-K(0)).
    \end{aligned} \right.
\end{equation}
Let us denote the new ODE system (\ref{reparam}) by \[(\Tilde{R}',\Tilde{\theta}') = f(\Tilde{R},\Tilde{\theta}).\]
For $\Tilde{R} > 0$, we can further simplify the ODE system by considering \[\Tilde{A} = \log\Tilde{R}.\] Then the ODE system (\ref{reparam}) can be rewritten into
\begin{equation} \label{reparam_log_1}
    \left\{ \begin{aligned}
        \Tilde{A}' &= 2(K'(0)-K'(-\Tilde{\theta}))\exp(\Tilde{A}) + 2(K'(\Tilde{\theta})-K'(0)),\\
        \Tilde{\theta}' &= 2(K(0)-K(-\Tilde{\theta}))\exp(\Tilde{A}) + 2(K(\Tilde{\theta})-K(0)).
    \end{aligned} \right.
\end{equation}
Denote the ODE system (\ref{reparam_log_1}) by \[(\Tilde{A}',\Tilde{\theta}') = g_1(\Tilde{A},\Tilde{\theta}).\]
For $\Tilde{R} < 0$, we can rather consider \[\Tilde{A} = \log(-\Tilde{R})\] and write
\begin{equation} \label{reparam_log_2}
    \left\{ \begin{aligned}
        \Tilde{A}' &= 2(K'(0)-K'(-\Tilde{\theta}))(-\exp(\Tilde{A})) + 2(K'(\Tilde{\theta})-K'(0)),\\
        \Tilde{\theta}' &= 2(K(0)-K(-\Tilde{\theta}))(-\exp(\Tilde{A})) + 2(K(\Tilde{\theta})-K(0)).
    \end{aligned} \right.
\end{equation}
Denote this system (\ref{reparam_log_2}) by \[(\Tilde{A}',\Tilde{\theta}') = g_2(\Tilde{A},\Tilde{\theta}).\]

Our systems possess the following symmetries:
\begin{itemize}
    \item $(I_1,I_2,\theta)$ and $(-I_1,-I_2,\theta)$. They correspond to the same $(\Tilde{R},\Tilde{\theta})$, but they have the opposite signs of the reparametrized time variable $s=\int I_2$. Thus, on the phase plane of (\ref{reparam}), they move along the same orbit but in opposite directions. Note that in view of (\ref{original}), $(-I_1,-I_2,\theta)$ with time variable $-t$ obeys the same ODE with $(I_1,I_2,\theta)$ with time variable $t$.
    \item $(\Tilde{R},\Tilde{\theta})$ and $(1/\Tilde{R},-\Tilde{\theta})$, for $\Tilde{R} > 0$. In this case, we can readily see that \[g_1(-\Tilde{A},-\Tilde{\theta}) = -\exp(-\Tilde{A})g_1(\Tilde{A},\Tilde{\theta}).\] This implies that on the phase plane of (\ref{reparam_log_1}), they move symmetrically with respect to $(0,\pi)$, but only with different velocities.
    \item $(\Tilde{R},\Tilde{\theta})$ and $(1/\Tilde{R},-\Tilde{\theta})$, for $\Tilde{R} < 0$. In this case, we similarly have \[g_2(-\Tilde{A},-\Tilde{\theta}) = \exp(-\Tilde{A})g_2(\Tilde{A},\Tilde{\theta}).\] This implies that on the phase plane of (\ref{reparam_log_2}), they move along two orbits which are symmetric with respect to $(0,\pi)$, but each point moves with the direction opposite to the symmetric one.
\end{itemize}
Note that the second and third cases correspond to the symmetry of $(I_1,I_2,\theta)$ and $(I_2,I_1,-\theta)$ in (\ref{original}). The interpretation of these symmetries will become more evident in the later discussions.

\section{Equilibria} \label{section:equil}
In this section, we obtain the following result about the equilibrium points of (\ref{reparam}).
\begin{theorem}[Equilibria of reparametrized system] \label{thm:equil}
    The reparametrized ODE system (\ref{reparam}) has the following equilibrium points, depending on the value of $\beta > 0$.
    \begin{itemize}
        \item $(1,\pi)$, an attractor if $0 < \beta < \beta^*$; a repeller if $\beta > \beta^*$.
        \item $(\Bar{R},\Bar{\theta})$ and $(1/\Bar{R},2\pi-\Bar{\theta})$, a pair of saddle points, for $0 < \beta < \beta^*$.
        \item $(0,2\pi)$, an attractor, for all $\beta > 0$.
        \item $(0,\theta_3)$, a repeller, for $0 < \beta < \beta_3$.
        \item $(0,\theta_2)$, a saddle point, for $0 < \beta < \beta_2$.
        \item $(0,\theta_1)$, a repeller, for $0 < \beta < \beta_0$.
        \item $(0,0)$, a saddle if $\beta \in (0,\beta_0)\cup(\beta_2,\beta_3)$; a repeller if $\beta \in (\beta_0,\beta_2)\cup(\beta_3,\infty)$.
        \item $(-1,0)$ and $(-1,2\pi)$, saddle points, for all $\beta > 0$.
    \end{itemize}
\end{theorem}
\begin{remark}
    Note that this result implies that the asymmetric self-similar spiral with two branches, whose existence was established in \cite{CKOb} and \cite{JS}, corresponding to $(\Bar{R},\Bar{\theta})$ or equivalently $(1/\Bar{R},2\pi-\Bar{\theta})$, is always unstable. On the other hand, the symmetric one (Alexander spiral), corresponding to $(1,\pi)$, is stable if the asymmetric one exists ($0 < \beta < \beta^*$), and unstable otherwise ($\beta > \beta^*$).
\end{remark}

\subsection{Nullclines and equilibrium points}
We first determine the locations of all equilibrium points of (\ref{reparam}). They are precisely the intersection points of two nullclines, $\{\Tilde{R}'=0\}$ and $\{\Tilde{\theta}'=0\}$. First, $\Tilde{R}' = 0$ is equivalent to $\Tilde{R} = 0$ or $\Tilde{R} = R_1(\Tilde{\theta})$, where $R_1$ is defined as \[R_1(\theta) = \frac{K'(\theta)-K'(0)}{K'(-\theta)-K'(0)}\] for each $\theta \neq 2\pi-\alpha$. (Note that $K'(-\theta) - K'(0)$ and $K'(\theta) - K'(0)$ cannot be both $0$ by Assumption \ref{assumption:K'(x)>K'(0)}.) Similarly, $\Tilde{\theta}' = 0$ is equivalent to $\Tilde{\theta} = 0, 2\pi$ or $\Tilde{R} = R_2(\Tilde{\theta})$, where $R_2$ is defined as \[R_2(\theta) = \frac{K(\theta)-K(0)}{K(-\theta)-K(0)}\] for each $\theta$ with $K(-\theta) \neq K(0)$. (Again, note that $K(-\theta) - K(0)$ and $K(\theta) - K(0)$ cannot be both $0$ unless $\theta = 0, 2\pi$ or $\beta = \beta_1, 1$, according to Assumption \ref{assumption:xi}.) Thus, each equilibrium point $(R_0,\theta_0)$ must satisfy one of the following:
\begin{itemize}
    \item $R_0 = R_1(\theta_0) = R_2(\theta_0)$.
    \item $R_0 = 0$ and $K(\theta_0) = K(0)$.
    \item $\theta_0 = 0, 2\pi$ and $R_0 = R_2(0) = R_2(2\pi) = -1$.
\end{itemize}
By Lemma \ref{lemma:K(x)=K(0)}, the points $(0,2\pi)$, $(0,\theta_3)$, $(0,\theta_2)$, $(0,\theta_1)$, and $(0,0)$, if each exists, compose the second case. Meanwhile, $R_1(\theta_0) = R_2(\theta_0)$ implies $F(\theta_0) = 0$, where $F$ is the function specified in (\ref{F(theta)}). Thus, by Assumption \ref{assumption:F(x)}, the first case consists of $(1,\pi)$, $(\Bar{R},\Bar{\theta})$, and $(1/\Bar{R},2\pi-\Bar{\theta})$ (the latter two exist if and only if $0 < \beta < \beta^*$), where $\Bar{R} = R_1(\Bar{\theta}) = R_2(\Bar{\theta})$. (Note that if $R_1(\theta)$ and $R_1(-\theta)$ are both defined, then $R_1(-\theta) = 1/R_1(\theta)$, and the same holds for $R_2$. This means that the intersection points of $R_1$ and $R_2$ must appear symmetrically.) So, we found all the equilibrium points.

Now, we sketch the graphs of the nullclines and check how they divide the whole phase plane and how they meet at the equilibrium points we found. This is useful because each divided region contains exactly one of the four directions($\nearrow, \nwarrow, \swarrow, \searrow$) of the vector field $f$. First, the graph of $R_1$ does not vary significantly with $\beta$. Assumption \ref{assumption:K'(x)>K'(0)} readily yields the following:
\begin{itemize}
    \item If $0 < \theta < 2\pi-\alpha$, then $R_1(\theta) < -1$.
    \item If $2\pi-\alpha < \theta < \alpha$, then $R_1(\theta) > 0$.
    \item If $\alpha < \theta < 2\pi$, then $-1 < R_1(\theta) < 0$.
\end{itemize}
What always holds is that the graphs of $R_1$ and $R_2$ cannot intersect in the region $R < 0$. Indeed, if $0 < \theta < 2\pi-\alpha$, then the mean value theorem and Assumption \ref{assumption:K'(x)>K'(0)} yield $K(-\theta) - K(0) > -K'(0)\theta > K(0) - K(\theta)$, which gives $R_2(\theta) > -1$. So the two graphs cannot meet in $\{0 < \theta < 2\pi-\alpha\}$, and in $\{\alpha < \theta < 2\pi\}$ by symmetry.

However, the graph of $R_2$ differs substantially, depending on which range $\beta$ belongs to. The function $R_2$ has vertical asymptotes $\theta = 2\pi-\theta_i$ ($i = 1, 2, 3$, depending on $\beta$), which separate the graph of $R_2$ into connected components. The value of $R_2$ jumps from $+\infty$ to $-\infty$, or from $-\infty$ to $+\infty$, when it passes through a vertical asymptote. Also, the sign of $R_2$ changes when it passes through $\theta_i$ ($i = 1, 2, 3$, again depending on $\beta$). Assumption \ref{assumption:xi} determines the order of these turning points. Finally, the sign of the leftmost part, $\lim_{\theta\to0+} R_2(\theta)$, is equal to that of $K'(+0)$, since $K'(-0) < 0$ by Lemma \ref{lemma:K'(+0)>K'(-0)}. We can sketch the graph of $R_2$ for each range of $\beta$ by considering all these.

\begin{figure}[t]
    \centering
    \begin{subfigure}{0.3\textwidth}
        \centering
        \includegraphics[width=\textwidth]{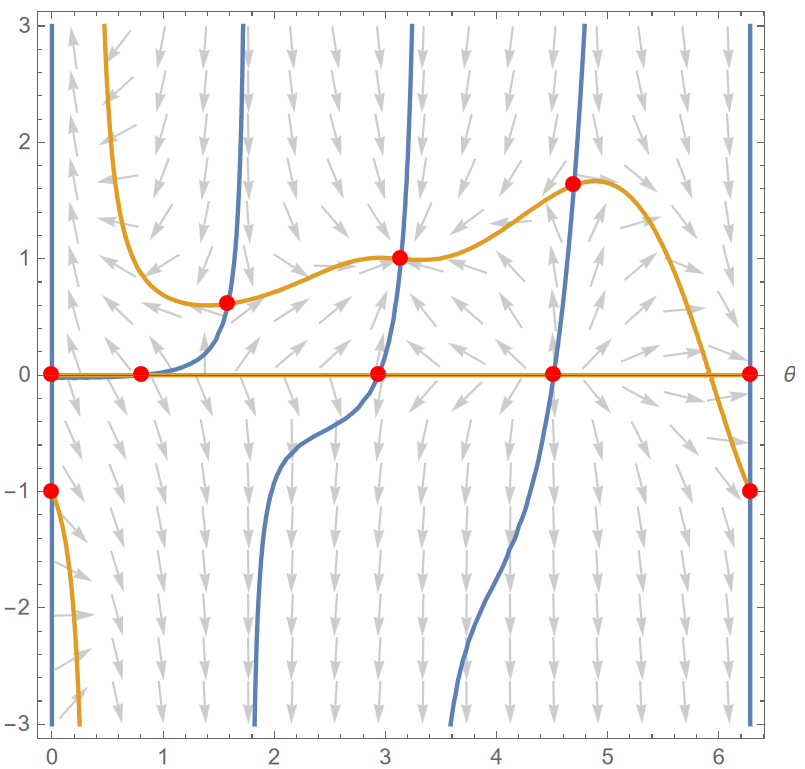}
        \caption{$\beta = 0.3 \in (0,\beta_0)$}
        \label{fig:phase_plane_0.3}
    \end{subfigure}
    \hfill
    \begin{subfigure}{0.3\textwidth}
        \centering
        \includegraphics[width=\textwidth]{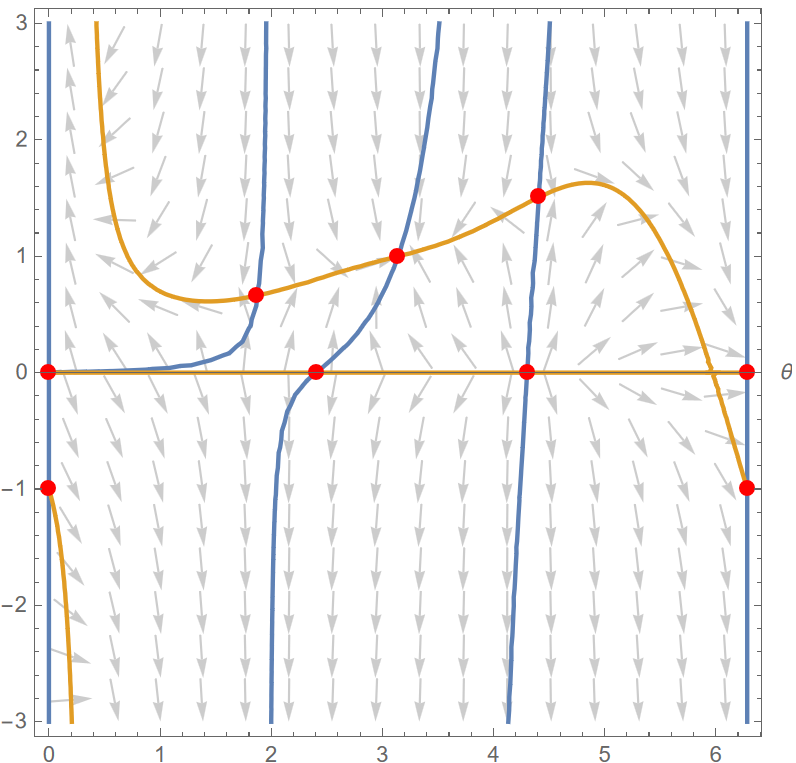}
        \caption{$\beta = 0.5 \in (\beta_0,\beta_1)$}
        \label{fig:phase_plane_0.5}
    \end{subfigure}
    \hfill
    \begin{subfigure}{0.3\textwidth}
        \centering
        \includegraphics[width=\textwidth]{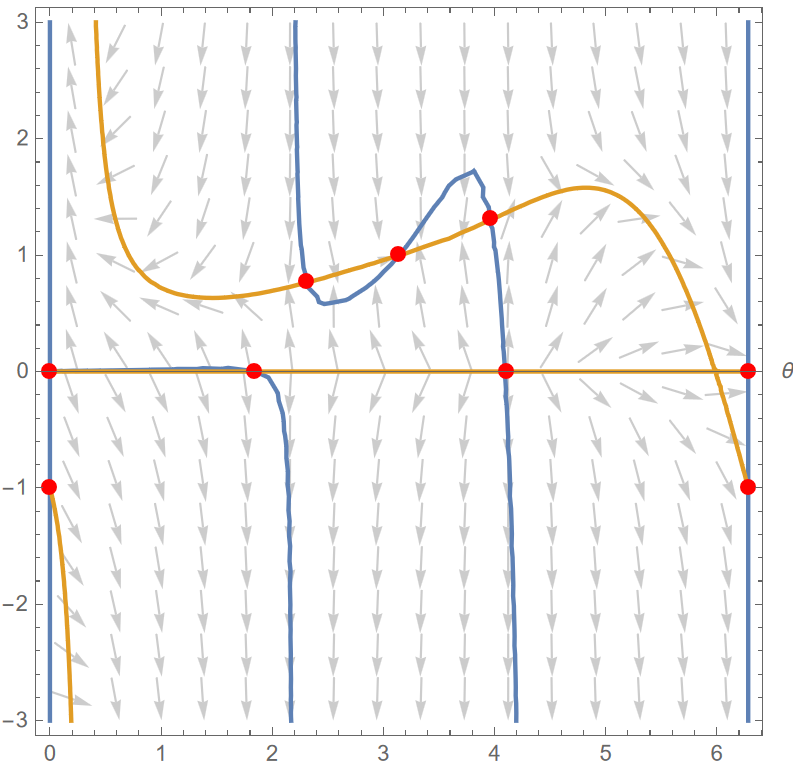}
        \caption{$\beta = 0.63 \in (\beta_1,\beta^*)$}
        \label{fig:phase_plane_0.63}
    \end{subfigure}
    \hfill
    \begin{subfigure}{0.3\textwidth}
        \centering
        \includegraphics[width=\textwidth]{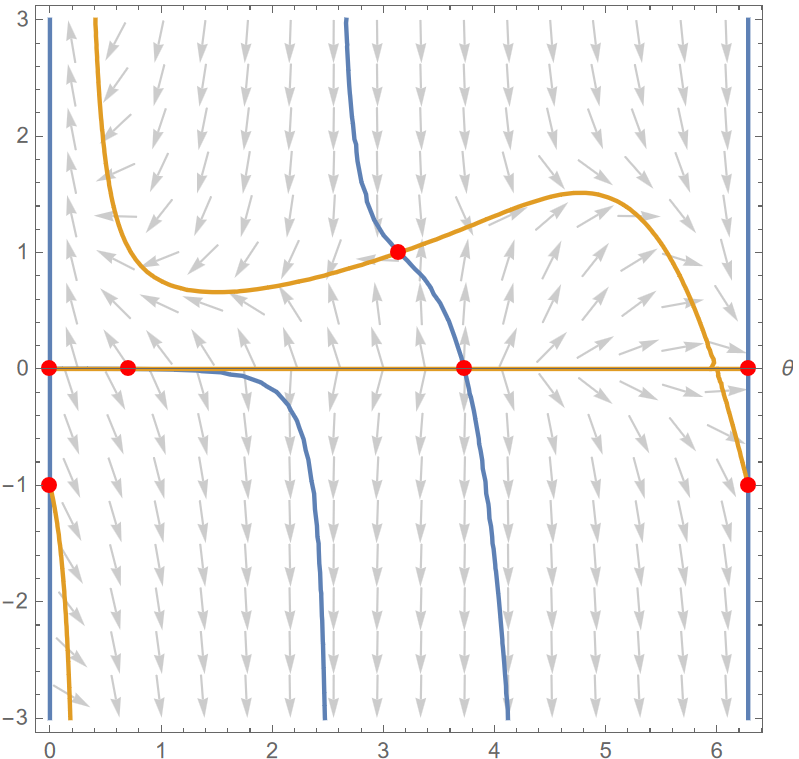}
        \caption{$\beta = 0.8 \in (\beta^*,\beta_2)$}
        \label{fig:phase_plane_0.8}
    \end{subfigure}
    \hfill
    \begin{subfigure}{0.3\textwidth}
        \centering
        \includegraphics[width=\textwidth]{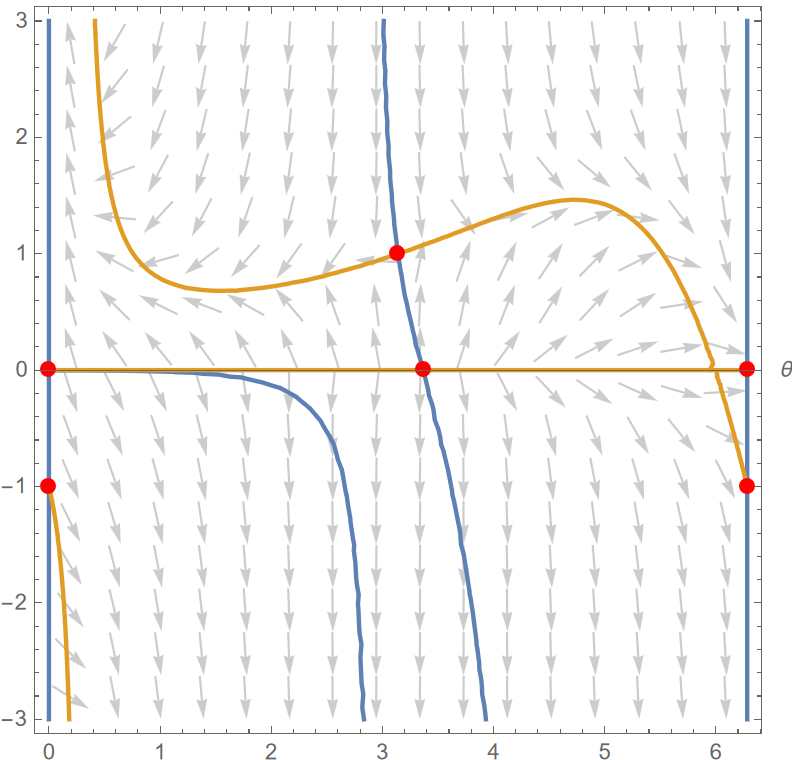}
        \caption{$\beta = 0.93 \in (\beta_2,1)$}
        \label{fig:phase_plane_0.93}
    \end{subfigure}
    \hfill
    \begin{subfigure}{0.3\textwidth}
        \centering
        \includegraphics[width=\textwidth]{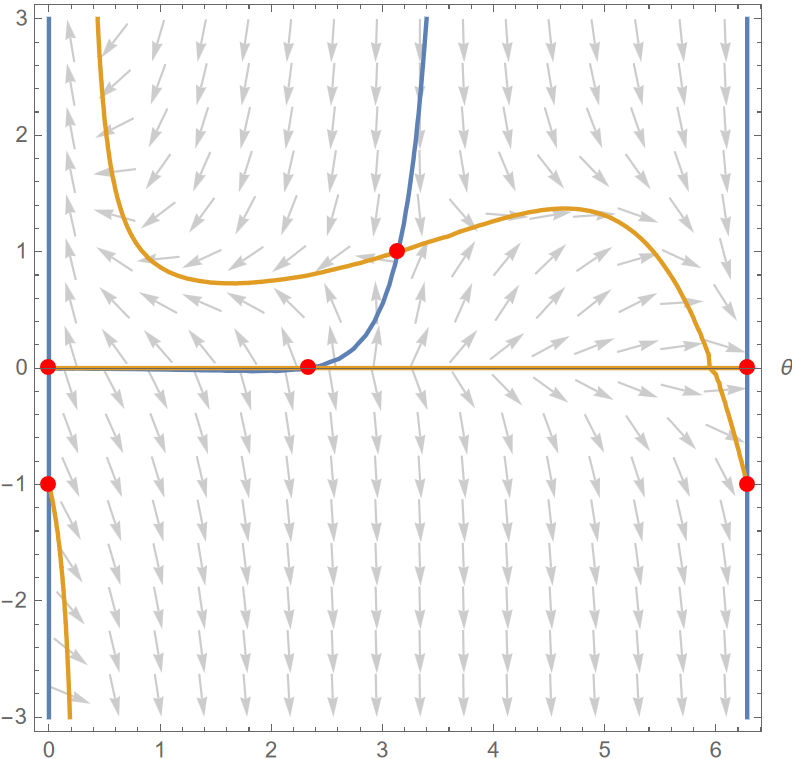}
        \caption{$\beta = 1.2 \in (1,\beta_3)$}
        \label{fig:phase_plane_1.2}
    \end{subfigure}
    \hfill
    \begin{subfigure}{0.3\textwidth}
        \centering
        \includegraphics[width=\textwidth]{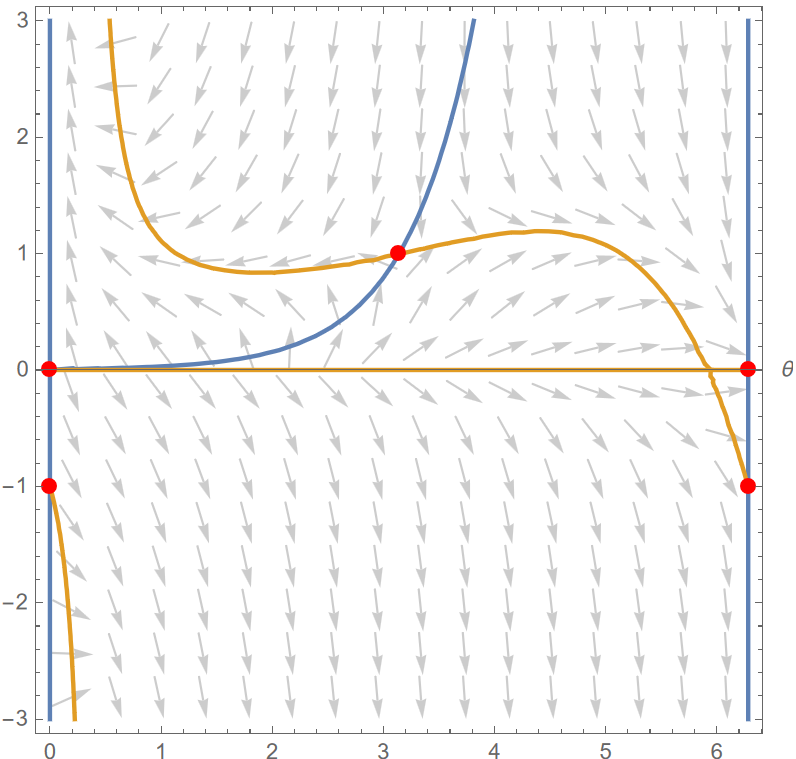}
        \caption{$\beta = 1.8 \in (\beta_3,\infty)$}
        \label{fig:phase_plane_1.8}
    \end{subfigure}
    \caption{Vector fields, nullclines, and equilibrium points of (\ref{reparam}) for several values of $\beta$}
    \label{fig:phase_planes}
\end{figure}

\textbf{Case 1: $\beta \in (0,\beta_0)$.} The graph of $R_2$ comprises four parts, separated by three vertical asymptotes. Since $K'(+0) < 0$, the first three parts cover all nonnegative values of $R$. Thus, each of them intersects with the positive part of $R_1$ at least once (note that this argument relies on the fact that $\theta_3 < \alpha$, which is verified in Section \ref{section:prop}), and exactly once, at $(\Bar{R},\Bar{\theta})$, $(1,\pi)$, and $(1/\Bar{R},2\pi-\Bar{\theta})$, respectively. Also, each of them meets $R = 0$ exactly once and forms an equilibrium point $(0,\theta_i)$ ($i = 1, 2, 3$). The fourth part does not form any equilibrium point. Figure \ref{fig:phase_plane_0.3} shows a typical phase plane of this case. (The figure does not show the fourth part of $R_2$.)

\textbf{Case 2: $\beta \in (\beta_0,\beta_1)$.} The graph of $R_2$ comprises three parts, separated by two vertical asymptotes. Since $K'(+0) > 0$ and $2\pi-\theta_3 < \theta_2$, the first part moves $R_2$ from some positive number to $+\infty$, not crossing $R = 0$, while the other two cover all real values of $R$, clearly crossing $R = 0$. But each still intersects with the positive part of $R_1$ exactly once. Figure \ref{fig:phase_plane_0.5} shows a typical phase plane of this case.

\textbf{Case 3: $\beta \in (\beta_1,\beta^*)$.} In this case, the graph of $R_2$ also comprises three parts. However, since $K'(+0) > 0$ and $\theta_2 < 2\pi-\theta_3$, the first part moves $R_2$ from some positive number to $-\infty$, crossing $R = 0$. Consequently, the second part moves $R_2$ from $+\infty$ to $-\infty$, intersecting with the positive part of $R_1$ at $(1,\pi)$ and crossing $R = 0$. Note that $F'(\pi) > 0$ in (\ref{F'(pi)}) implies \[R_1'(\pi) = \frac{2K''(\pi)}{K'(\pi)-K'(0)} < \frac{2K'(\pi)}{K(\pi)-K(0)} = R_2'(\pi),\] since $K'(\pi) - K'(0) > 0$ and $K(\pi) - K(0) < 0$. Thus, the second part must meet $R_1$ twice more, one in $\theta < \pi$ and the other in $\theta > \pi$. So, the three equilibrium points $(1,\pi)$, $(\Bar{R},\Bar{\theta})$, and $(1/\Bar{R},2\pi-\Bar{\theta})$ all belong to the second part. The third part does not form any equilibrium point. Figure \ref{fig:phase_plane_0.63} shows a typical phase plane of this case. (The figure does not show the third part of $R_2$.)

\textbf{Case 4: $\beta \in (\beta^*,\beta_2)$.} The only difference of this case from the one right before is that the two additional intersection points of $R_1$ and the second part of $R_2$ do not exist. Figure \ref{fig:phase_plane_0.8} shows a typical phase plane of this case. (Again, the figure does not show the third part of $R_2$.)

\textbf{Case 5: $\beta \in (\beta_2,1)$.} The graph of $R_2$ comprises two parts, separated by a vertical asymptote. Since $K'(+0) < 0$ and $2\pi-\theta_3 < \theta_3$, the first part only covers negative values of $R$. On the other hand, the second part meets the positive part of $R_1$ at $(1,\pi)$ and crosses $R = 0$ at $(0,\theta_3)$. Figure \ref{fig:phase_plane_0.93} shows a typical phase plane of this case.

\textbf{Case 6: $\beta \in (1,\beta_3)$.} This case is essentially identical to the case right before. The only difference is that $\theta_3 < 2\pi-\theta_3$ so the equilibrium points $(0,\theta_3)$ and $(1,\pi)$ both belong to the first part. But this does not make any fundamental change, as we will see in Section \ref{section:longtime}. Figure \ref{fig:phase_plane_1.2} shows a typical phase plane of this case. (The figure does not show the second part of $R_2$.)

\textbf{Case 7: $\beta \in (\beta_3,\infty)$.} The graph of $R_2$ is connected and meets the positive part of $R_1$ at $(1,\pi)$. It does not cross $R = 0$. Figure \ref{fig:phase_plane_1.8} shows a typical phase plane of this case.

\begin{remark}
    Because of (\ref{-beta}), the phase plane for $-\beta$ is obtained by simply flipping the one for $\beta$.
\end{remark}

\subsection{Behavior near the equilibria}
Now, we determine the local behavior near the equilibrium points specified in Theorem \ref{thm:equil}. Recall the Hartman-Grobman theorem: if the Jacobian matrix
\begin{equation} \label{Jacobian}
    Df(R_0,\theta_0) = \begin{bmatrix}
    4(K'(0)-K'(-\theta_0))R_0 + 2(K'(\theta_0)-K'(0)) & 2K''(-\theta_0)R_0^2 + 2K''(\theta_0)R_0\\
    2(K(0)-K(-\theta_0)) & 2K'(-\theta_0)R_0 + 2K'(\theta_0)
    \end{bmatrix}
\end{equation}
of (\ref{reparam}) at an equilibrium point $(R_0,\theta_0)$ has two eigenvalues with nonzero real parts, then $(R_0,\theta_0)$ is called hyperbolic and the orbit structure of the system in a neighborhood of the equilibrium is topologically equivalent to that of the linearized one. This means that the local behavior is determined by the signs of the real parts of the two eigenvalues. The two eigenvalues are the solutions to the characteristic equation
\begin{equation} \label{char_eq}
    \lambda^2 - a\lambda + b = 0
\end{equation}
of (\ref{Jacobian}), where $a,b$ are real. If $b<0$, then (\ref{char_eq}) has two real solutions, one positive and one negative, so $(R_0,\theta_0)$ is a saddle point. Otherwise, if $b>0$, then (\ref{char_eq}) has two real solutions or two complex solutions that are complex conjugates. In any case, $a$ is the sum of the real parts of the two solutions. Thus, if $a<0$ and $b>0$, then $(R_0,\theta_0)$ is an attractor; if $a,b>0$, then it is a repeller.

\textbf{Case 1: $(R_0,\theta_0) = (1,\pi)$.} The Jacobian matrix (\ref{Jacobian}) at this point equals
\begin{equation} \label{Df(1,pi)}
    Df(1,\pi) = \begin{bmatrix}
        2(K'(0)-K'(\pi)) & 4K''(\pi)\\
        2(K(0)-K(\pi)) & 4K'(\pi)
    \end{bmatrix}.
\end{equation}
Consider the characteristic equation (\ref{char_eq}) of (\ref{Df(1,pi)}). Then \[a = 2(K'(0)+K'(\pi))\] and \[b = 8\bigl[K'(\pi)(K'(0)-K'(\pi))-K''(\pi)(K(0)-K(\pi))\bigr].\]
From (\ref{F'(pi)}), we see that \[b = 4F'(\pi).\]
According to Lemma \ref{lemma:K'(0)<0}, $a$ remains negative for all $\beta > 0$. On the other hand, the sign of $b$ depends on $\beta$, by Assumption \ref{assumption:F(x)}. If $0 < \beta < \beta^*$, then $b$ is positive and $(1,\pi)$ is an attractor. On the other hand, if $\beta > \beta^*$, then $b$ is negative and $(1,\pi)$ is a saddle point.

\textbf{Case 2: $(R_0, \theta_0) = (\Bar{R}, \Bar{\theta}), (1/\Bar{R}, 2\pi-\Bar{\theta})$.} These two points are symmetric, so we only consider $(\Bar{R}, \Bar{\theta})$. Here, it is more convenient to consider the logarithmically rescaled system (\ref{reparam_log_1}) rather than (\ref{reparam}). The scaling does not change the essential local behavior. The Jacobian matrix of (\ref{reparam_log_1}) at $(\Tilde{A},\Tilde{\theta}) = (\log\Bar{R},\Bar{\theta})$ is
\begin{equation} \label{Jacobian_log}
    Dg_1(\log\Bar{R},\Bar{\theta}) = \begin{bmatrix}
    2(K'(0)-K'(-\Bar{\theta}))\Bar{R} & 2K''(-\Bar{\theta})\Bar{R} + 2K''(\Bar{\theta})\\
    2(K(0)-K(-\Bar{\theta}))\Bar{R} & 2K'(-\Bar{\theta})\Bar{R} + 2K'(\Bar{\theta})
    \end{bmatrix}.
\end{equation}
If we express the characteristic equation of (\ref{Jacobian_log}) as (\ref{char_eq}), then the definition of $\Bar{R}$ and Assumption \ref{assumption:F(x)} yield
\begin{align*}
    b/4\Bar{R} ={}& (K'(0)-K'(-\Bar{\theta}))K'(-\Bar{\theta})\Bar{R} + (K'(0)-K'(-\Bar{\theta}))K'(\Bar{\theta})\\
    & - (K(0)-K(-\Bar{\theta}))K''(-\Bar{\theta})\Bar{R} - (K(0)-K(-\Bar{\theta}))K''(\Bar{\theta})\\
    ={}& (K'(0)-K'(\Bar{\theta}))K'(-\Bar{\theta}) + (K'(0)-K'(-\Bar{\theta}))K'(\Bar{\theta})\\
    & - (K(0)-K(\Bar{\theta}))K''(-\Bar{\theta}) - (K(0)-K(-\Bar{\theta}))K''(\Bar{\theta})\\
    ={}& F'(\Bar{\theta}) < 0.
\end{align*}
Thus, $(\Bar{R},\Bar{\theta})$ is a saddle point. Due to symmetry, $(1/\Bar{R}, 2\pi-\Bar{\theta})$ is also a saddle point.

\textbf{Case 3: $(R_0, \theta_0) = (0, \theta)$, where $K(\theta) = K(0)$.} In other words, $\theta$ is one of $0$, $\theta_1$, $\theta_2$, $\theta_3$, and $2\pi$, depending on $\beta$. In this case, the Jacobian matrix (\ref{Jacobian}) is
\begin{equation} \label{Df(0,x)}
    Df(0,\theta) = \begin{bmatrix}
        2(K'(\theta)-K'(0)) & 0\\
        2(K(0)-K(-\theta)) & 2K'(\theta)
    \end{bmatrix}.
\end{equation}
If we express the characteristic equation of (\ref{Df(0,x)}) as (\ref{char_eq}), then \[a = 4K'(\theta)-2K'(0)\] and \[b = 4K'(\theta)(K'(\theta)-K'(0)).\] If $\theta = 2\pi$, then $a<0$ and $b>0$, since $K'(-0) < K'(0) < 0$ by Lemma \ref{lemma:K'(+0)>K'(-0)}. So, $(0,2\pi)$ is an attractor. Otherwise, we have $K'(\theta) > K'(0)$, according to Assumption \ref{assumption:K'(x)>K'(0)} and the fact that $\theta_3 < \alpha$. Thus, if $K'(\theta)<0$, then $b<0$ so $(0,\theta)$ is a saddle point; if $K'(\theta)>0$, then $a,b>0$ so $(0,\theta)$ is a repeller. Finally, Lemma \ref{lemma:K'(x):sign} determines the local behavior. (Note that if $(0,\theta)$ is a saddle point, its stable manifold always lies in $R = 0$.)

\textbf{Case 4: $(R_0, \theta_0) = (-1, 0), (-1, 2\pi)$.} Since these two points are antisymmetric, we only consider $(-1,2\pi)$. The Jacobian matrix (\ref{Jacobian}) at this point equals
\begin{equation} \label{Df(1,2pi)}
    Df(-1,2\pi) = \begin{bmatrix}
        K'(+0)-K'(-0) & 2(K''(+0)-K''(-0))\\
        0 & 2(K'(-0)-K'(+0))
    \end{bmatrix}.
\end{equation}
If we express the characteristic equation of (\ref{Df(1,2pi)}) as (\ref{char_eq}), then \[b = -2(K'(-0)-K'(+0))^2 < 0.\] Thus, $(-1,2\pi)$ is a saddle point. The orbits near $(-1,0)$ has the opposite directions to the symmetric ones near $(-1,2\pi)$, so $(-1,0)$ is also a saddle point.

\section{Long-time behavior of solutions} \label{section:longtime}
In this section, we fully classify the long-time behavior of the spirals with two branches and analyze their asymptotics.

\subsection{Non-existence of cycles and convergence to the equilibria}
Recall the celebrated Poincaré-Bendixson theorem: in a 2D ODE system with finitely many equilibrium points, a bounded solution must converge to an equilibrium point, a limit cycle, or a cycle composed of homoclinic and heteroclinic orbits. Now, we show no such cycle exists, so bounded solutions always converge to an equilibrium.

\begin{theorem}[Non-existence of cycles] \label{thm:no_cycles}
    Suppose $|\beta| \notin \{\beta_0, \beta^*, \beta_2, \beta_3\}$ so all the equilibrium points of (\ref{reparam}) are hyperbolic. In the phase plane of (\ref{reparam}), all kinds of cycles, namely closed trajectories, homoclinic orbits, and heteroclinic cycles, do not exist. Therefore, any bounded solution to (\ref{reparam}) converges to some equilibrium point.
\end{theorem}
\begin{proof}
    For convenience, let us use the notation $R,\theta$ instead of $\Tilde{R},\Tilde{\theta}$. First, note that $R = 0$ is an invariant set, so any closed or homoclinic orbit must be fully contained in either $R > 0$ or $R < 0$. This is also true for any heteroclinic cycle since all saddle points on $R = 0$ repel trajectories in $R \neq 0$. (In other words, their stable manifolds always lie in $R = 0$.) Thus, we can consider two regions $R > 0$ and $R < 0$ separately.

    We use the index theory illustrated in \cite[Corollary~1.8.5]{GH83}. Although the theory directly applies only to closed orbits, we can also use it to prove the non-existence of homoclinic and heteroclinic cycles with a following perturbation argument. Suppose we have a homoclinic orbit $\gamma$ joining a hyperbolic saddle point $p$ to itself. Then, we can take a small enough neighborhood of $p$ containing no other equilibrium points and perturb the inside vector field to obtain a new smooth closed orbit $\gamma'$, which has the same index number as $\gamma$. Likewise, we can make a closed orbit from a heteroclinic cycle by perturbation in each neighborhood of hyperbolic saddle points of the cycle.
    
    According to the theory, within any cycle, there must be an attractor or a repeller since all the equilibrium points are hyperbolic. Then, the lower region $R < 0$ does not contain cycles since it does not contain an equilibrium point in its interior. If $|\beta| > \beta^*$, inside the upper region $R > 0$ there is only one equilibrium point $(1,\pi)$, which is a saddle, so the upper region also does not have cycles.

    Now suppose $0 < \beta < \beta^*$ and there is a cycle in $R > 0$. Since $(1,\pi)$ is the only non-saddle equilibrium point in this region, the index theory yields that the cycle must enclose this point. For the case $0 < \beta \leq \beta_1$, the region $\theta_2 \leq \theta \leq 2\pi-\theta_2$ is an invariant set in $R > 0$ since $\theta' \geq 0$ on $\theta = \theta_2$ and $\theta' \leq 0$ on $\theta = 2\pi-\theta_2$. Thus, the whole cycle must be contained in this region. Now, convert the phase plane into that of the log-scaled system (\ref{reparam_log_1}), which is symmetrical with respect to $(1,\pi)$. Then the cycle is still contained in $\theta_2 \leq \theta \leq 2\pi-\theta_2$. The uppermost point $N$ and the lowermost point $S$ of the cycle are on the log-scaled graph of $R_1$, while the leftmost point $W$ and the rightmost point $E$ are on that of $R_2$. We note that the absolute value of the slope of $\overline{NS}$ is always greater than that of $\overline{EW}$. By the mean value theorem, the slope of $\overline{NS}$ equals \[(\log R_1)'(\theta) = \frac{K''(\theta)}{K'(\theta)-K'(0)} + \frac{K''(-\theta)}{K'(-\theta)-K'(0)}\] for some $\theta \in (\theta_2,2\pi-\theta_2)$, while that of $\overline{EW}$ equals \[(\log R_2)'(\eta) = \frac{K'(\eta)}{K(\eta)-K(0)} + \frac{K'(-\eta)}{K(-\eta)-K(0)}\] for some $\eta \in (\theta_2,2\pi-\theta_2)$. Then Assumption \ref{assumption:slope} yields $|(\log R_1)'(\theta)| < (\log R_2)'(\eta)$, a contradiction.

    Finally, we consider the case $\beta \in (\beta_1,\beta^*)$ and again only focus on the upper half-plane $R > 0$. In this case, the region $\{0 \leq \theta \leq 2\pi-\theta_3\} \cup \{2\pi-\theta_3 < \theta \leq \Bar{\theta}, R \leq R_2(\theta)\}$ and its symmetric region are invariant sets, since $R' < 0$ on the graph of $R_2$ on $2\pi-\theta_3 < \theta < \Bar{\theta}$ and $\theta' < 0$ on $\{(R,\Bar{\theta}) : R<\Bar{R}\}$. Thus, the cycle must be contained in the region in between. Also, by Assumption \ref{assumption:slope}, the slope of $R_1$ in $\Bar{\theta} < \theta < 2\pi-\Bar{\theta}$ is always positive, which implies that the region enclosed by the graphs of $R_1$ and $R_2$ is an invariant set. However, the cycle must pass through this region, again a contradiction.
\end{proof}

Not all solutions to (\ref{reparam}) are bounded. However, unbounded solutions are symmetric to bounded ones if we consider the symmetry of $(\Tilde{R},\Tilde{\theta})$ and $(1/\Tilde{R},2\pi-\Tilde{\theta})$ (with time reversal when $\Tilde{R} < 0$) discussed previously because all unbounded solutions must have $|\Tilde{R}| \to \infty$. Indeed, for any large $C>0$, each unbounded orbit must enter the region where $|\Tilde{R}|$ is greater than $C$ and increasing, the boundary of which consists of $|\Tilde{R}| = C$, $\Tilde{\theta} = 0,2\pi$, and the graph of $R_1$ near $\Tilde{\theta} = 2\pi-\alpha$. The slope of the graph is negative for large enough $C$, since \[R_1'(\theta) = \frac{K''(\theta)(K'(-\theta)-K'(0))+K''(-\theta)(K'(\theta)-K'(0))}{(K'(-\theta)-K'(0))^2}\] is negative if $\theta$ is close to $2\pi-\alpha$ and, accordingly, $K'(-\theta)-K'(0)$ is small enough. (Note that Assumption \ref{assumption:K'(x)>K'(0)} yields $K''(\alpha)(K'(-\alpha)-K'(0)) < 0$.) This makes the region invariant, implying $|\Tilde{R}| \to \infty$ for unbounded solutions. In other words, we get $\Tilde{R} \to 0$ for the symmetric ones, hence the following corollary.

\begin{corollary} \label{cor:convergence}
    For each solution to (\ref{reparam}), either itself or the symmetric one (in an appropriate sense) converges to an equilibrium point.
\end{corollary}

\subsection{Solutions to the reparametrized system}
As a result of Corollary \ref{cor:convergence}, we can partition the entire phase plane so that all the orbits in each region converge to a single equilibrium point (if we consider $(\pm\infty,\theta)$ as a single point by compactification). This can be done by drawing a topological graph indicating the heteroclinic orbits between the equilibria. Applying Corollary \ref{cor:convergence} to both directions of time, we see that every nontrivial orbit becomes a heteroclinic orbit by taking $s \to \pm\infty$.

\begin{figure}[t]
    \centering
    \begin{subfigure}{0.3\textwidth}
        \centering
        \includegraphics[width=\textwidth]{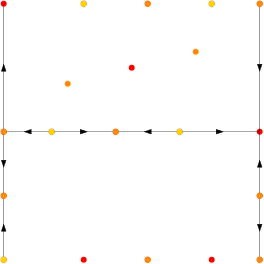}
        \caption{}
        \label{fig:top_graphs_11}
    \end{subfigure}
    \hfill
    \begin{subfigure}{0.3\textwidth}
        \centering
        \includegraphics[width=\textwidth]{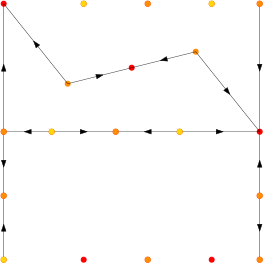}
        \caption{}
        \label{fig:top_graphs_12}
    \end{subfigure}
    \hfill
    \begin{subfigure}{0.3\textwidth}
        \centering
        \includegraphics[width=\textwidth]{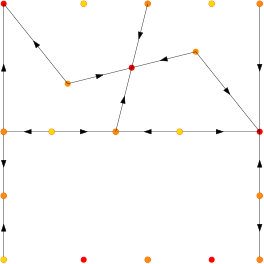}
        \caption{}
        \label{fig:top_graphs_13}
    \end{subfigure}
    \hfill
    \begin{subfigure}{0.3\textwidth}
        \centering
        \includegraphics[width=\textwidth]{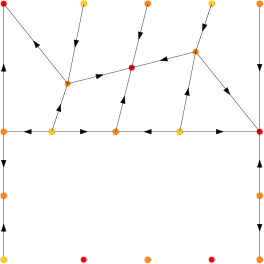}
        \caption{}
        \label{fig:top_graphs_14}
    \end{subfigure}
    \hfill
    \begin{subfigure}{0.3\textwidth}
        \centering
        \includegraphics[width=\textwidth]{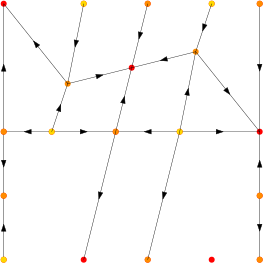}
        \caption{}
        \label{fig:top_graphs_15}
    \end{subfigure}
    \hfill
    \begin{subfigure}{0.3\textwidth}
        \centering
        \includegraphics[width=\textwidth]{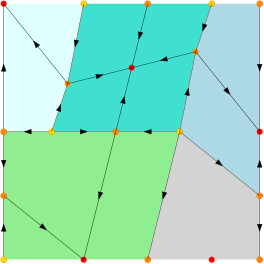}
        \caption{}
        \label{fig:top_graphs_16}
    \end{subfigure}
    \caption{Constructing the topological graph of equilibrium points in the case of $\beta \in (0,\beta_0)$}
    \label{fig:top_graphs_construction}
\end{figure}

Consider the case $\beta \in (0,\beta_0)$ and recall its typical phase plane, Figure \ref{fig:phase_plane_0.3}. We first mark all the equilibrium points and also their symmetric ones, namely $(\pm\infty,2\pi-\theta)$ where $K(\theta) = K(0)$. We can readily find the heteroclinic orbits on the invariant sets $\Tilde{R} = 0$, $\Tilde{\theta} = 0$, and $\Tilde{\theta} = 2\pi$, as Figure \ref{fig:top_graphs_11} indicates. In the figure, the attractors, the saddle points, and the repellers are marked red, orange, and yellow, respectively. The attractors have only inward edges, while the repellers have only outward edges. The saddle points have one-dimensional stable and unstable manifolds, so they have two (or one if they are on the boundary) inward edges and two (or one) outward edges. (This does not apply to the `$\infty$ points'.) By symmetry discussed before, $(+\infty,2\pi-\theta)$ has the same local behavior with $(0,\theta)$ in $\Tilde{R} > 0$, while that of $(-\infty,2\pi-\theta)$ is the opposite. We can see that there are no more edges towards the points on $\Tilde{R} = 0$ except $(0,2\pi)$, so no edge can head towards $+\infty$ points except $(+\infty,0)$. Meanwhile, every edge cannot cross the other edges.

Now, notice the saddle point $(\Bar{R},\Bar{\theta})$, which has two edges starting from it. They cannot enter the region $\Tilde{\theta} > 2\pi-\theta_2$ since $\Tilde{\theta}' < 0$ in between the second part of the graph of $R_2$ and $\Tilde{\theta} = 2\pi-\theta_2$. Thus, the only possibility is heading towards $(+\infty,0)$ and $(1,\pi)$, as indicated in Figure \ref{fig:top_graphs_12}. By symmetry, we can draw the two edges starting from $(1/\Bar{R},2\pi-\Bar{\theta})$.

Next, we focus on the unique outward edge of $(0,\theta_2)$ in $\Tilde{R} > 0$. While it still cannot enter $\Tilde{\theta} > 2\pi-\theta_2$, it also cannot enter $\Tilde{\theta} < 2\pi-\theta_3$ since $\Tilde{\theta}' > 0$ in between the first part of the graph of $R_2$ and $\Tilde{\theta} = 2\pi-\theta_3$. Thus, the only possibility is heading towards $(1,\pi)$, as indicated in Figure \ref{fig:top_graphs_13} with the symmetric edge. Then, the starting points of two inward edges of $(\Bar{R},\Bar{\theta})$ are automatically determined, as expressed in Figure \ref{fig:top_graphs_14}.

There also exists a unique outward edge of $(0,\theta_2)$ in $\Tilde{R} < 0$. Since it cannot enter $\Tilde{\theta} > \theta_2$, the only option is reaching the attractor $(-\infty,2\pi-\theta_3)$, as expressed in Figure \ref{fig:top_graphs_15}. By antisymmetry, there exists an edge from $(0,\theta_3)$ to $(-\infty,2\pi-\theta_2)$. Finally, the only possible destination of the unique outward edge of $(-1,0)$ is $(-\infty,2\pi-\theta_3)$, and the antisymmetric edge is from $(0,\theta_3)$ to $(-1,2\pi)$, completing the construction of the graph. Figure \ref{fig:top_graphs_16} indicates the partition of the whole phase plane, each region being attracted to the unique attractor (marked red) in that region. Each borderline between the two regions is attracted to the unique saddle point (marked orange) on it.

\begin{figure}[t]
    \centering
    \begin{subfigure}{0.3\textwidth}
        \centering
        \includegraphics[width=\textwidth]{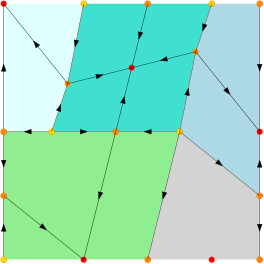}
        \caption{$\beta \in (0,\beta_0)$}
        \label{fig:top_graphs_1}
    \end{subfigure}
    \hfill
    \begin{subfigure}{0.3\textwidth}
        \centering
        \includegraphics[width=\textwidth]{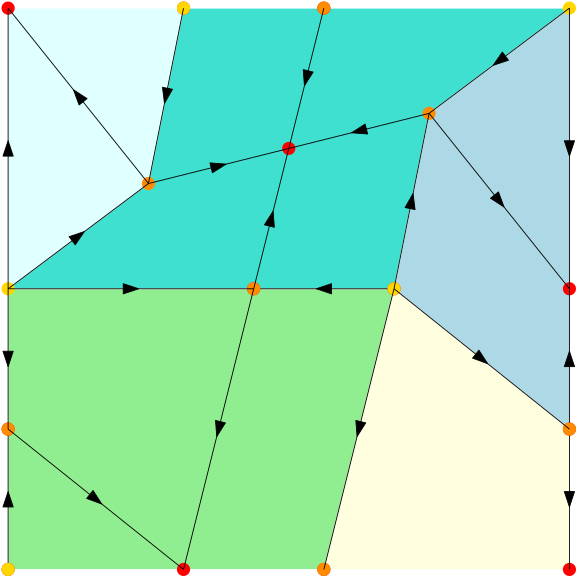}
        \caption{$\beta \in (\beta_0,\beta_1)$}
        \label{fig:top_graphs_2}
    \end{subfigure}
    \hfill
    \begin{subfigure}{0.3\textwidth}
        \centering
        \includegraphics[width=\textwidth]{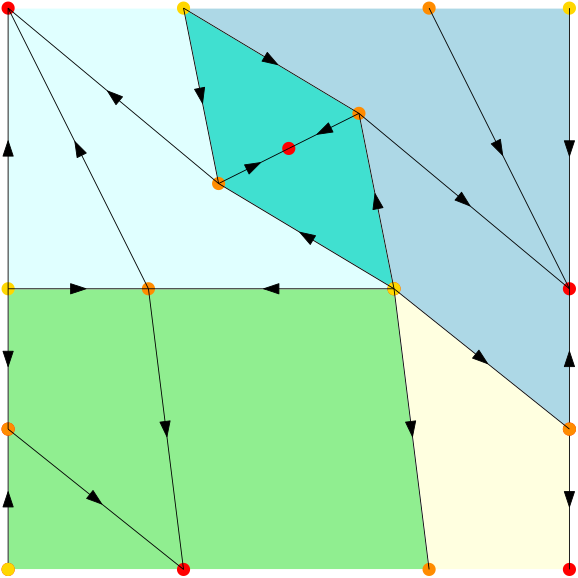}
        \caption{$\beta \in (\beta_1,\beta^*)$}
        \label{fig:top_graphs_3}
    \end{subfigure}
    \hfill
    \begin{subfigure}{0.3\textwidth}
        \centering
        \includegraphics[width=\textwidth]{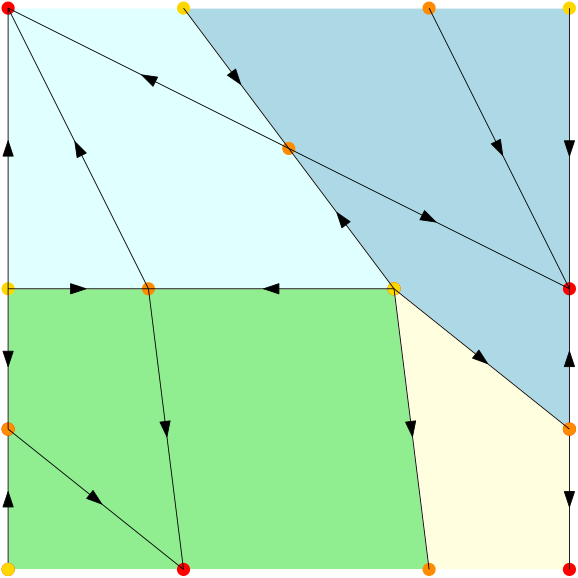}
        \caption{$\beta \in (\beta^*,\beta_2)$}
        \label{fig:top_graphs_4}
    \end{subfigure}
    \hfill
    \begin{subfigure}{0.3\textwidth}
        \centering
        \includegraphics[width=\textwidth]{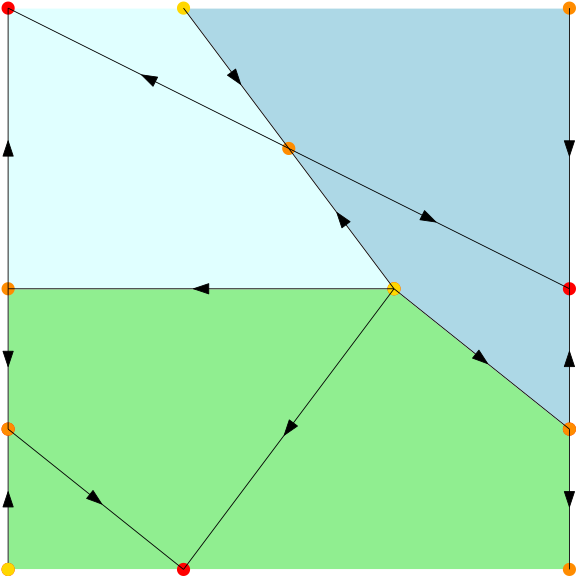}
        \caption{$\beta \in (\beta_2,1)$}
        \label{fig:top_graphs_5}
    \end{subfigure}
    \hfill
    \begin{subfigure}{0.3\textwidth}
        \centering
        \includegraphics[width=\textwidth]{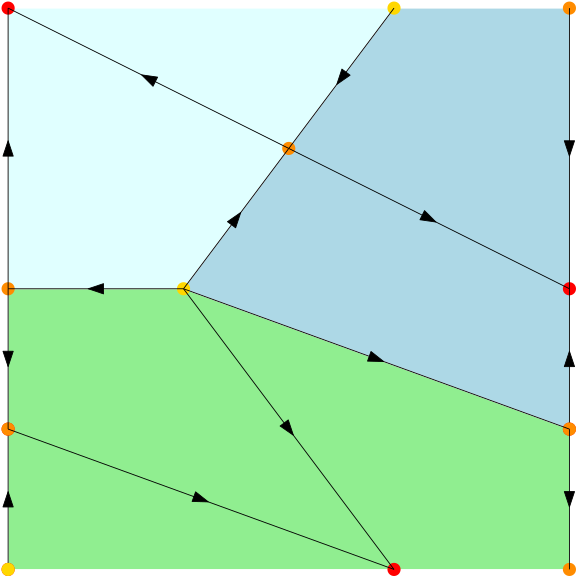}
        \caption{$\beta \in (1,\beta_3)$}
        \label{fig:top_graphs_6}
    \end{subfigure}
    \hfill
    \begin{subfigure}{0.3\textwidth}
        \centering
        \includegraphics[width=\textwidth]{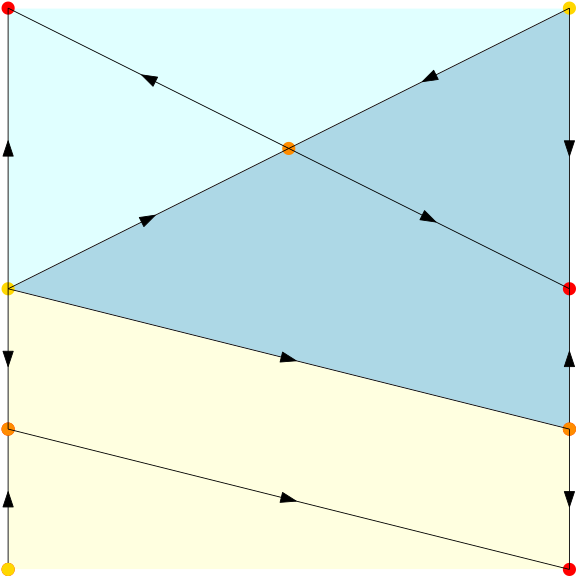}
        \caption{$\beta \in (\beta_3,\infty)$}
        \label{fig:top_graphs_7}
    \end{subfigure}
    \caption{Topological graphs of heteroclinic orbits and partitions of the phase planes}
    \label{fig:top_graphs}
\end{figure}

The other cases are only simpler, so we can similarly obtain the graphs for all cases as Figure \ref{fig:top_graphs}. Note that the graphs of the cases $\beta \in (\beta_2,1)$ and $\beta \in (1,\beta_3)$ are equivalent. The graphs for the negative direction of the time variable $s$ can be obtained by simply reversing the edges.

We also note that Figure \ref{fig:top_graphs} suggests the asymptotic stability of the symmetric and asymmetric self-similar spirals, determined in Theorem \ref{thm:equil} and the remark below it, is natural. In the absence of the asymmetric spiral, only $(0,2\pi)$, $(+\infty,0)$, and $(1,\pi)$ attract orbits in $\Tilde{R} > 0$. Since $(0,2\pi)$ is always an attractor, its basin of attraction is open (and so is that of $(+\infty,0)$). If $(1,\pi)$ were also an attractor, then its basin of attraction would also be open, contradicting the connectedness of $\Tilde{R} > 0$. Thus, in this case, $(1,\pi)$ must be unstable. On the other hand, when an asymmetric spiral exists, then the symmetric one need not be unstable. Meanwhile, given the uniqueness of the asymmetric spiral, a similar topological argument yields its instability.

We can summarize the results as the following theorem.
\begin{theorem}[Long-time behavior of reparametrized system] \label{thm:longtime_reparam}
    Suppose $\beta > 0$, $\beta \notin \{\beta_0, \beta^*, \beta_2, \beta_3\}$, $\Tilde{R} \neq 0$, and $0 < \Tilde{\theta} < 2\pi$. For the positive direction of $s$, any solution $(\Tilde{R},\Tilde{\theta})$ to (\ref{reparam}) approaches one of the following:
    \begin{itemize}
        \item $(1,\pi)$, as $s\to+\infty$;
        \item $(\Bar{R},\Bar{\theta})$, as $s\to+\infty$;
        \item $(1/\Bar{R},2\pi-\Bar{\theta})$, as $s\to+\infty$;
        \item $(0,2\pi)$, as $s\to+\infty$;
        \item $(+\infty,0)$, as $s\uparrow s^*$ for some $s^*<+\infty$;
        \item $(-1,2\pi)$, as $s\to+\infty$;
        \item $(-\infty,\theta_0)$, where $K(-\theta_0) = 0$, as $s\uparrow s^*$ for some $s^*<+\infty$.
    \end{itemize}
    Meanwhile, for the negative direction of $s$, the solution approaches one of the following:
    \begin{itemize}
        \item $(1,\pi)$, as $s\to-\infty$;
        \item $(\Bar{R},\Bar{\theta})$, as $s\to-\infty$;
        \item $(1/\Bar{R},2\pi-\Bar{\theta})$, as $s\to-\infty$;
        \item $(0,\theta_0)$, where $K(\theta_0) = 0$, as $s\to-\infty$;
        \item $(+\infty,2\pi-\theta_0)$, where $K(\theta_0) = 0$, as $s\downarrow -s^*$ for some $s^*<+\infty$;
        \item $(-1,0)$, as $s\to-\infty$;
        \item $(-\infty,0)$, as $s\downarrow -s^*$ for some $s^*<+\infty$.
    \end{itemize}
\end{theorem}

\subsection{Solutions to the original system}
Suppose we have a solution $(\Tilde{R}(s), \Tilde{\theta}(s))$ to the new ODE system (\ref{reparam}) and initial data for the original ODE system (\ref{original}). Then, we can recover the corresponding solution $(I_1(t), I_2(t), \theta(t))$ for the original ODE system as follows. First, recall that \[I_2' = 2K'(0)I_2^2 + 2K'(-\theta)I_1I_2.\] Suppose $I_2 > 0$ and divide both sides by $I_2$ to get \[(\log I_2)' = 2K'(0)I_2 + 2K'(-\theta)I_1.\] If we divide both sides by $I_2$ again, we have
\begin{equation} \label{I2_tilde}
    (\log\Tilde{I_2})' = 2K'(0) + 2K'(-\Tilde{\theta})\Tilde{R},
\end{equation}
where $\Tilde{I_2} = I_2\circ\phi^{-1}$. Since the functions $\Tilde{R}$ and $\Tilde{\theta}$ are given, we can determine the function $\Tilde{I_2}$ from (\ref{I2_tilde}). We can also figure out the function $\Tilde{I_1} = I_1\circ\phi^{-1}$ at the same time, since $\Tilde{I_1} / \Tilde{I_2} = \Tilde{R}$. Finally, by solving the ODE $\phi' = \Tilde{I_2}\circ\phi$, we can find out the function $\phi$, and simultaneously $I_1$, $I_2$, and $\theta$ by composition of $\phi$ with $\Tilde{I_1}$, $\Tilde{I_2}$, and $\Tilde{\theta}$, respectively. The formula for $\phi$ is
\begin{equation} \label{phi}
    \int_0^{\phi(t)} \frac{ds}{\Tilde{I_2}(s)} = t.
\end{equation}
We apply this recovering strategy to each case of Theorem \ref{thm:longtime_reparam} to get our main result, Theorem \ref{thm:longtime_original}.

\textbf{Case 1: $\lim_{s\to+\infty}(\Tilde{R}(s),\Tilde{\theta}(s)) = (1,\pi)$.} In this case, (\ref{I2_tilde}) and Lemma \ref{lemma:K'(0)<0} give $(\log\Tilde{I_2}(s))' \to 2K'(0)+2K'(\pi) < 0$ as $s \to +\infty$. So, $\log\Tilde{I_2}(s) \to -\infty$ and $\Tilde{I_2}(s) \to +0$ as $s \to +\infty$. Also, since $K'$ and $\Tilde{R}$ are bounded, $(\log\Tilde{I_2})'$ is also bounded and $\Tilde{I_2}(s) \in (0,+\infty)$ for all $s \geq 0$. Thus, (\ref{phi}) implies that $\phi(t) \to +\infty$ as $t \to +\infty$. By composition, we see that \[I_2(t)\to+0, \quad R(t)=\frac{I_1(t)}{I_2(t)}\to1, \quad \theta(t)\to\pi \quad \text{as } t\to+\infty.\] Moreover, from \[\left(\frac{1}{I_2}\right)' = -2K'(0) - 2K'(-\theta)R \to -2K'(0) - 2K'(\pi),\] we get \[I_1(t) \sim I_2(t) \sim \frac{1}{-2(K'(0)+K'(\pi))t}\] by L'Hôpital's rule.

\textbf{Case 2: $\lim_{s\to+\infty}(\Tilde{R}(s),\Tilde{\theta}(s)) = (\Bar{R},\Bar{\theta})$.} In this case, (\ref{I2_tilde}) gives $(\log\Tilde{I_2}(s))' \to 2K'(0)+2K'(-\Bar{\theta})\Bar{R}$ as $s \to \infty$. The limit is again negative, since \[K'(0)+K'(-\Bar{\theta})\Bar{R} = \frac{K'(\Bar{\theta})K'(-\Bar{\theta})-K'(0)^2}{K'(-\Bar{\theta})-K'(0)}\] and Lemma \ref{lemma:K'(0)<0} implies that the numerator is negative when the denominator is positive due to Assumption \ref{assumption:K'(x)>K'(0)} and the fact that $2\pi-\Bar{\theta} < \alpha$. Thus, through similar reasoning, we get \[I_2(t)\to+0, \quad R(t)=\frac{I_1(t)}{I_2(t)}\to\Bar{R}, \quad \theta(t)\to\Bar{\theta} \quad \text{as } t\to+\infty.\] Determining the asymptotic behavior is also similar, which yields \[I_1(t) \sim \frac{K'(\Bar{\theta})-K'(0)}{-2(K'(\Bar{\theta})K'(-\Bar{\theta})-K'(0)^2)t}\] and \[I_2(t) \sim \frac{K'(-\Bar{\theta})-K'(0)}{-2(K'(\Bar{\theta})K'(-\Bar{\theta})-K'(0)^2)t}.\]

\textbf{Case 2$'$: $\lim_{s\to+\infty}(\Tilde{R}(s),\Tilde{\theta}(s)) = (1/\Bar{R},2\pi-\Bar{\theta})$.} This case is symmetric to the one right before in the sense that if we swap $(I_1,I_2)$ and $(\theta,2\pi-\theta)$, we can move from one case to the other. Thus, the asymptotics of $I_1$ and $I_2$ are also swapped.

\textbf{Case 3: $\lim_{s\to+\infty}(\Tilde{R}(s),\Tilde{\theta}(s)) = (0,2\pi)$.} Since $\Tilde{R} \to 0$ and $K'$ is bounded, (\ref{I2_tilde}) and Lemma \ref{lemma:K'(0)<0} yield $(\log\Tilde{I_2}(s))' \to 2K'(0) < 0$ as $s \to +\infty$. Through similar reasoning, we get \[I_2(t)\to+0, \quad R(t)=\frac{I_1(t)}{I_2(t)}\to0, \quad \theta(t)\to-0 \quad \text{as } t\to+\infty\] and \[I_2(t) \sim \frac{1}{-2K'(0)t}.\] Moreover, from \[\frac{(\log|I_1|)'}{I_2} = 2K'(0)R + 2K'(\theta) \to 2K'(-0),\] we obtain \[\log|I_1|(t) \sim -\frac{K'(-0)}{K'(0)}\log t\] by L'Hôpital's rule.

\textbf{Case 3$'$: $\lim_{s\uparrow s^*}(\Tilde{R}(s),\Tilde{\theta}(s)) = (+\infty,0)$, where $s^*<+\infty$.} This case is symmetric to the one right before, so we have \[I_1(t) \sim \frac{1}{-2K'(0)t}\] and \[\log I_2(t) \sim -\frac{K'(-0)}{K'(0)}\log t.\] Note that the asymptotic behavior of $\log I_2$ implies the boundedness of $s = \int I_2$.

\textbf{Case 4: $\lim_{s\to+\infty}(\Tilde{R}(s),\Tilde{\theta}(s)) = (-1,2\pi)$.} In this case, (\ref{I2_tilde}) and Lemma \ref{lemma:K'(+0)>K'(-0)} give $(\log\Tilde{I_2}(s))' \to 2K'(0)-2K'(+0) < 0$ as $s \to +\infty$. Through similar reasoning, we get \[I_2(t)\to+0, \quad R(t)=\frac{I_1(t)}{I_2(t)}\to-1, \quad \theta(t)\to-0 \quad \text{as } t\to+\infty.\] Moreover, we obtain \[I_1(t) \sim \frac{1}{2(K'(0)-K'(+0))t}\] and \[I_2(t) \sim \frac{1}{-2(K'(0)-K'(+0))t}.\]

\textbf{Case 5: $\lim_{s\uparrow s^*}(\Tilde{R}(s),\Tilde{\theta}(s)) = (-\infty,\theta_0)$, where $s^* < +\infty$ and $K(-\theta_0) = 0$.} In this case, \[\left(\frac{1}{\Tilde{R}}\right)' = -2(K'(0)-K'(-\Tilde{\theta})) - 2(K'(\Tilde{\theta})-K'(0))\frac{1}{\Tilde{R}} \to -2(K'(0)-K'(-\theta_0)) > 0\] as $s\uparrow s^*$ by Assumption \ref{assumption:K'(x)>K'(0)}. Thus, by L'Hôpital's rule, the asymptotic behavior of $\Tilde{R}$ as $s\uparrow s^*$ is \[\Tilde{R}(s) \sim \frac{1}{-2(K'(0)-K'(-\theta_0))(s-s^*)}.\] Then (\ref{I2_tilde}) yields \[(\log\Tilde{I_2})' \sim -\frac{K'(-\theta_0)}{K'(0)-K'(-\theta_0)} \frac{1}{s-s^*}.\] Thus, applying L'Hôpital's rule again gives \[\log\Tilde{I_2}(s) \sim -\frac{K'(-\theta_0)}{K'(0)-K'(-\theta_0)}\log(s^*-s),\] where \[-\frac{K'(-\theta_0)}{K'(0)-K'(-\theta_0)} \in (-\infty,1).\] This implies \[\int_0^{s^*} \frac{1}{\Tilde{I_2}(s)}ds < +\infty.\] Thus, by (\ref{phi}), we have $\phi(t) \uparrow s^*$ as $t \uparrow t^*$ for some $t^* < +\infty$. Then, \[\quad \frac{1}{R(t)}=\frac{I_2(t)}{I_1(t)}\to-0, \quad \theta(t)\to\theta_0 \quad \text{as } t\uparrow t^*.\] Moreover, \[\left(\frac{1}{I_1}\right)' = -2K'(0) - 2K'(\theta)\frac{1}{R} \to -2K'(0)\] yields \[I_1(t) \sim \frac{1}{-2K'(0)(t-t^*)}\] as $t\uparrow t^*$. Then, \[\frac{(\log I_2)'}{I_1} \to 2K'(-\theta_0)\] gives \[\log I_2(t) \sim -\frac{K'(-\theta_0)}{K'(0)}\log(t^*-t).\] Note that $I_2 \to +0$ if $K'(-\theta_0) > 0$ but $I_2 \to +\infty$ if $K'(-\theta_0) < 0$.

The other cases, corresponding to $I_2<0$, can be handled similarly. According to the symmetry, we can consider $(-I_1,-I_2,\theta)$ instead of $(I_1,I_2,\theta)$. This means we again suppose $I_2>0$ but instead move $t$ (and accordingly $s$) in the negative direction.

\textbf{Case 6: $\lim_{s\to-\infty}(\Tilde{R}(s),\Tilde{\theta}(s)) = (1,\pi)$.} From (\ref{I2_tilde}), we have $\log\Tilde{I_2}(s) \sim 2(K'(0)+K'(\pi))s$ as $s\to-\infty$, which implies $\lim_{s\to-\infty}\Tilde{I_2}(s) = +\infty$ and also \[\int_0^{-\infty} \frac{1}{\Tilde{I_2}(s)}ds > -\infty,\] since $K'(0)+K'(\pi) < 0$. Then (\ref{phi}) yields $\phi(t)\to-\infty$ as $t\downarrow -t^*$ for some $t^*<+\infty$. By composition, we have \[I_2(t)\to+\infty, \quad R(t)=\frac{I_1(t)}{I_2(t)}\to1, \quad \theta(t)\to\pi \quad \text{as } t\downarrow -t^*.\] Moreover, an analogous argument to the case $s\to+\infty$ yields \[I_1(t) \sim I_2(t) \sim \frac{1}{-2(K'(0)+K'(\pi))(t+t^*)}\] as $t\downarrow -t^*$.

\textbf{Case 7: $\lim_{s\to-\infty}(\Tilde{R}(s),\Tilde{\theta}(s)) = (\Bar{R},\Bar{\theta})$.} Computing analogously to the case right before, we have \[I_1(t) \sim \frac{K'(\Bar{\theta})-K'(0)}{-2(K'(\Bar{\theta})K'(-\Bar{\theta})-K'(0)^2)(t+t^*)}\] and \[I_2(t) \sim \frac{K'(-\Bar{\theta})-K'(0)}{-2(K'(\Bar{\theta})K'(-\Bar{\theta})-K'(0)^2)(t+t^*)}\] as $t\downarrow -t^*$ for some $t^*<+\infty$.

\textbf{Case 7$'$: $\lim_{s\to-\infty}(\Tilde{R}(s),\Tilde{\theta}(s)) = (1/\Bar{R},2\pi-\Bar{\theta})$.} Due to symmetry with the case right before, the asymptotics of $I_1$ and $I_2$ are swapped.

\textbf{Case 8: $\lim_{s\to-\infty}(\Tilde{R}(s),\Tilde{\theta}(s)) = (0,\theta_0)$.} From (\ref{I2_tilde}), we have $\log\Tilde{I_2}(s) \sim 2K'(0)s$ as $s\to-\infty$. Since $K'(0)<0$, we again obtain $\phi(t)\to-\infty$ as $t\downarrow -t^*$ for some $t^*<+\infty$ and \[I_2(t)\to+\infty, \quad R(t)=\frac{I_1(t)}{I_2(t)}\to0, \quad \theta(t)\to\theta_0 \quad \text{as } t\downarrow -t^*.\] Moreover, an analogous argument to the case $\lim_{s\uparrow s^*}(\Tilde{R}(s),\Tilde{\theta}(s)) = (-\infty,\theta_0)$ yields \[I_2(t) \sim \frac{1}{-2K'(0)(t+t^*)}, \quad \log |I_1|(t) \sim -\frac{K'(\theta_0)}{K'(0)}\log(t+t^*)\] as $t\downarrow -t^*$. Note that $K'(-\theta_0)$ might be negative, which then implies $|I_1|(t) \to +\infty$.

\textbf{Case 8$'$: $\lim_{s\downarrow-s^*}(\Tilde{R}(s),\Tilde{\theta}(s)) = (+\infty,2\pi-\theta_0)$.} This case is symmetric to the case right before, so we get \[I_1(t) \sim \frac{1}{-2K'(0)(t+t^*)}, \quad \log I_2(t) \sim -\frac{K'(\theta_0)}{K'(0)}\log(t+t^*)\] as $t\downarrow -t^*$ for some $t^*<+\infty$. Note that the asymptotic behavior of $\log I_2$ implies the boundedness of $s = \int I_2$ as $t \downarrow -t^*$.

\textbf{Case 9: $\lim_{s\to-\infty}(\Tilde{R}(s),\Tilde{\theta}(s)) = (-1,0)$.} In this case, (\ref{I2_tilde}) yields \[\lim_{s\to-\infty}(\log\Tilde{I_2}(s))' = 2K'(0)-2K'(-0) > 0.\] This implies $\lim_{s\to-\infty}\Tilde{I_2}(s) = +0$ and $\phi(t) \to -\infty$ as $t \to -\infty$. Accordingly, we have \[I_2(t)\to+0, \quad R(t)=\frac{I_1(t)}{I_2(t)}\to-1, \quad \theta(t)\to0 \quad \text{as } t\to -\infty.\] Moreover, from \[\left(\frac{1}{I_2}\right)' = -2K'(0) - 2K'(-\theta)R \to -2K'(0) + 2K'(-0),\] we get \[I_2(t) \sim \frac{1}{-2(K'(0)-K'(-0))t}\] and \[I_1(t) \sim \frac{1}{2(K'(0)-K'(-0))t}\] as $t\to-\infty$.

\textbf{Case 10: $\lim_{s\downarrow-s^*}(\Tilde{R}(s),\Tilde{\theta}(s)) = (-\infty,0)$.} The argument is essentially analogous to that of the case $\lim_{s\uparrow s^*}(\Tilde{R}(s),\Tilde{\theta}(s)) = (-\infty,\theta_0)$. Now, \[\left(\frac{1}{\Tilde{R}}\right)' = -2(K'(0)-K'(-\Tilde{\theta})) - 2(K'(\Tilde{\theta})-K'(0))\frac{1}{\Tilde{R}} \to -2(K'(0)-K'(-0)) < 0\] as $s\downarrow -s^*$. By L'Hôpital's rule, \[\Tilde{R}(s) \sim \frac{1}{-2(K'(0)-K'(-0))(s+s^*)}.\] Then (\ref{I2_tilde}) yields \[(\log\Tilde{I_2})' \sim -\frac{K'(-0)}{K'(0)-K'(-0)} \frac{1}{s+s^*},\] where, in this case, \[-\frac{K'(-0)}{K'(0)-K'(-0)} > 1.\] Applying L'Hôpital's rule again gives \[\log\Tilde{I_2}(s) \sim -\frac{K'(-0)}{K'(0)-K'(-0)}\log(s+s^*).\] This implies $\lim_{s\downarrow -s^*}\Tilde{I_2}(s) = +0$ and also \[\int_0^{-s^*} \frac{1}{\Tilde{I_2}(s)}ds = -\infty.\] Thus, by (\ref{phi}), we have $\phi(t) \downarrow -s^*$ as $t\to-\infty$. Then, \[I_2(t)\to+0, \quad \frac{1}{R(t)}=\frac{I_2(t)}{I_1(t)}\to-0, \quad \theta(t)\to0 \quad \text{as } t\to-\infty.\] Moreover, \[\left(\frac{1}{I_1}\right)' = -2K'(0) - 2K'(\theta)\frac{1}{R} \to -2K'(0)\] yields \[I_1(t) \sim \frac{1}{-2K'(0)t}\] as $t\to-\infty$. Then, \[\frac{(\log I_2)'}{I_1} \to 2K'(-0)\] gives \[\log I_2(t) \sim -\frac{K'(-0)}{K'(0)}\log(-t).\]

Each case corresponds to one of the cases in Theorem \ref{thm:longtime_original}, so we are done.


\begin{thebibliography}{9}
\bibitem{Ale71}
R. C. Alexander. Family of similarity flows with vortex sheets. \emph{The Physics of Fluids}, 14(2):231–239, 1971.
\bibitem{Bir62}
G. Birkhoff. Helmholtz and Taylor instability. \emph{Proceedings of Symposia in Applied Mathematics}, XIII:55–76, 1962.
\bibitem{CKOa}
Tomasz Cieślak, Piotr Kokocki, and Wojciech S. Ożański. Well-posedness of logarithmic spiral vortex sheets. \emph{arXiv preprint arXiv:2110.07543}, 2021.
\bibitem{CKOb}
Tomasz Cieślak, Piotr Kokocki, and Wojciech S. Ożański. Existence of nonsymmetric logarithmic spiral vortex sheet solutions to the 2D Euler equations. \emph{arXiv preprint arXiv:2207.06056}, 2022.
\bibitem{CKOc}
Tomasz Cieślak, Piotr Kokocki, and Wojciech S. Ożański. Linear instability of symmetric logarithmic spiral vortex sheets. \emph{arXiv preprint arXiv:2305.08764}, 2023.
\bibitem{GH83}
John Guckenheimer and Philip Holmes. \emph{Nonlinear Oscillations, Dynamical Systems, and Bifurcations of Vector Fields,} volume 42 of \emph{Applied Mathematical Sciences.} Springer-Verlag, New York, 1983.
\bibitem{JS}
In-Jee Jeong and Ayman Said. Logarithmic spirals in 2d perfect fluids. \emph{arXiv preprint arXiv:2302.09447}, 2023.
\bibitem{Pra22}
L. Prandtl. Über die entstehung von wirbeln in der idealen flüssigkeit, mit anwendung auf die tragflügeltheorie und andere aufgaben. \emph{Vorträge aus dem Gebiete der Hydro- und Aerodynamik (Innsbruck)}, pages 18–33, 1922.
\bibitem{Rot56}
Nicholas Rott. Diffraction of a weak shock with vortex generation. \emph{Journal of Fluid Mechanics}, 1(1):111–128, 1956.
\end{thebibliography}
\end{document}